\documentclass[11pt]{article}

\usepackage[a4paper,margin=1in]{geometry}
\usepackage{amsmath,amssymb,amsthm,mathtools}
\usepackage{microtype}
\usepackage{enumitem}
\usepackage{booktabs}
\usepackage{graphicx}
\usepackage{aliascnt}
\usepackage{xcolor}
\usepackage[colorlinks=true,linkcolor=blue!55!black,citecolor=blue!55!black,urlcolor=blue!55!black]{hyperref}
\usepackage[nameinlink,capitalise,noabbrev]{cleveref}

\usepackage{tikz}
\usetikzlibrary{arrows.meta,calc}

\newtheorem{theorem}{Theorem}[section]
\newaliascnt{lemma}{theorem}
\newtheorem{lemma}[lemma]{Lemma}
\aliascntresetthe{lemma}
\newaliascnt{proposition}{theorem}
\newtheorem{proposition}[proposition]{Proposition}
\aliascntresetthe{proposition}
\newaliascnt{corollary}{theorem}
\newtheorem{corollary}[corollary]{Corollary}
\aliascntresetthe{corollary}
\newaliascnt{claim}{theorem}

\aliascntresetthe{claim}
\newaliascnt{problem}{theorem}
\newtheorem{problem}[problem]{Problem}
\aliascntresetthe{problem}
\theoremstyle{definition}
\newaliascnt{definition}{theorem}
\newtheorem{definition}[definition]{Definition}
\aliascntresetthe{definition}
\newaliascnt{construction}{theorem}
\newtheorem{construction}[construction]{Construction}
\aliascntresetthe{construction}
\theoremstyle{remark}
\newtheorem*{remark}{Remark}

\newcommand{\eps}{\varepsilon}
\newcommand{\cA}{\mathcal A}
\newcommand{\cC}{\mathcal C}

\newcommand{\cF}{\mathcal F}

\newcommand{\cK}{\mathcal K}
\newcommand{\cP}{\mathcal P}

\newcommand{\cT}{\mathcal T}

\newcommand{\ext}{\operatorname{ext}}

\newcommand{\thr}{\vartheta}

\title{Tight Hamilton Cycles in Linearly Quasirandom 3-Graphs}
\author{Xichao Shu\thanks{Institute of Mathematics, Leipzig University, Leipzig, Germany. Supported by the Alexander von Humboldt Foundation in the framework of the Alexander von Humboldt Professorship of Daniel Kr{\'a}l' endowed by the Federal Ministry of Education and Research. Email: \texttt{xichao.shu@uni-leipzig.de}.}}
\date{}

\begin{document}
\maketitle
\begin{abstract}
We study tight Hamilton cycles in linearly quasirandom $3$-graphs.  An
$n$-vertex $3$-graph $H$ is $(p,\mu)$-dense if $e_H(X,Y,Z)\ge p|X||Y||Z|-\mu n^3$ for all $X,Y,Z\subseteq V(H)$. 
Ara\'ujo, Piga and Schacht asked whether the conditions
$p,\alpha>1/4$ and $\delta_2(H)\ge\alpha n$ force a tight Hamilton cycle. We give a negative answer to this
question.  More generally, we determine the asymptotically sharp
minimum-codegree threshold for the existence of a tight Hamilton cycle
for every density $p\in(0,1)$.  The resulting threshold is a
discontinuous piecewise-defined function with four distinct regimes,
and matching constructions show that every piece is best possible.
The proof combines the absorption method and a fixed-length connecting
lemma above density $1/3$ with a canonical-component Hamilton framework
at and below density $1/3$.
\end{abstract}

\section{Introduction}
\label{sec:introduction}

\subsection{Background}
Tur{\'a}n-type extremal problems form a central theme in extremal
combinatorics: given a fixed hypergraph $F$, one asks how many edges an
$F$-free hypergraph can have. 
For a $k$-graph $F$, let $\operatorname{ex}_k(n,F)$ denote the maximum
number of edges in an $F$-free $k$-graph on $n$ vertices, and let
\[
    \pi(F)
    :=
    \lim_{n\to\infty}
    \frac{\operatorname{ex}_k(n,F)}{\binom nk}
\]
be its Tur{\'a}n density.  For graphs, the classical results of Mantel
and Tur{\'a}n determine the extremal numbers of complete graphs
\cite{Mantel1907,Turan1941}, while the Erd{\H{o}}s--Stone theorem
determines the Tur{\'a}n density of every graph
\cite{ErdosStone1946}.

The corresponding problems for uniform hypergraphs are substantially
more difficult.  A central and longstanding example is Tur{\'a}n's
tetrahedron problem, which asks for the Tur{\'a}n density of the
complete $3$-graph $K_4^{(3)}$.  Tur{\'a}n conjectured that
\[
    \pi(K_4^{(3)})=\frac59,
\]
but this remains open~\cite{Turan1941,Keevash2011Survey}. Erd{\H{o}}s offered \$500 for determining
$\pi(K_t^{(k)})$ for any fixed $t>k>2$, and \$1000 for resolving the
whole family of complete-hypergraph Tur{\'a}n problems
\cite{Erdos1981}.  For further background on hypergraph Tur{\'a}n problems, see
Keevash~\cite{Keevash2011Survey}.

Many standard lower-bound constructions for hypergraph Tur{\'a}n
problems are blow-ups of finite configurations and contain
linear-sized vertex sets of very low induced density.  Erd{\H{o}}s and
S{\'o}s~\cite{ES2} proposed excluding this source of inhomogeneity by
imposing hereditary density conditions on the host hypergraph.  This
led to the notion of uniform Tur{\'a}n density. A major exact result was obtained independently by Glebov,
Kr{\'a}l' and Volec~\cite{GlebovKralVolec2016} and by Reiher,
R{\"o}dl and Schacht~\cite{ReiherRodlSchacht2018WeakQR}: both groups
proved that the uniform Tur{\'a}n density of $K_4^{(3)-}$ is $1/4$,
where $K_4^{(3)-}$ denotes the tetrahedron with one edge removed.
Reiher,
R{\"o}dl and Schacht~\cite{ReiherRodlSchacht2018Vanishing} subsequently characterised the $3$-graphs of
uniform Tur{\'a}n density zero and showed that every positive uniform
Tur{\'a}n density is at least $1/27$.  Garbe, Kr{\'a}l' and Lamaison~\cite{GarbeKralLamaison2024}
later showed that this lower bound is attained.  In another direction, Buci{\'c}, Cooper,
Kr{\'a}l', Mohr and Munh{\'a} Correia~\cite{BucicCooperKralMohrMunhaCorreia2023} determined the uniform
Tur{\'a}n densities of tight cycles in $3$-uniform hypergraphs. Further results and open
problems are discussed in the survey of
Reiher~\cite{Reiher2020Survey}.

These hereditary density conditions belong to the broader theory of
hypergraph quasirandomness.  For graphs, Chung, Graham and Wilson showed
that several natural quasirandomness properties are equivalent
\cite{ChungGrahamWilson1989}.  This equivalence breaks down for
hypergraphs, where several genuinely different notions arise; see, for
example,
\cite{Ch90,Ch91,Chung10,ConlonHanPersonSchacht2012,
LenzMubayi_eig,lenz2015poset,Towsner,quasihyper}.
The hereditary density condition appearing in the definition of uniform
Tur{\'a}n density may be viewed as a one-sided form of linear
quasirandomness.  In this paper we use the following formulation, which
is particularly convenient for monotone containment problems.

\begin{definition}[$(p,\mu)$-density]
\label{def:dense}
Let $k\ge2$ and let $0<p,\mu<1$.  A $k$-graph $H$ on $n$ vertices is
called \emph{$(p,\mu)$-dense} if, for all
$X_1,\ldots,X_k\subseteq V(H)$,
\[
    e_H(X_1,\ldots,X_k)
    \ge
    p|X_1|\cdots|X_k|-\mu n^k,
\]
where $e_H(X_1,\ldots,X_k)$ denotes the number of ordered $k$-tuples
$(x_1,\ldots,x_k)\in X_1\times\cdots\times X_k$ such that
$\{x_1,\ldots,x_k\}\in E(H)$.
\end{definition}

Linear quasirandomness is naturally adapted to linear hypergraphs, in which no two edges meet
in more than one vertex. The results of
Kohayakawa, Nagle, R{\"o}dl and Schacht~\cite{KohayakawaNagleRodlSchacht2010} and of Conlon, H{\`a}n,
Person and Schacht~\cite{ConlonHanPersonSchacht2012} show that this condition is closely connected with
counting labelled copies of every fixed linear hypergraph.  Non-linear spanning configurations are
considerably more delicate, since their edges repeatedly overlap in
larger sets.

We study a spanning extension of the above extremal questions.  Rather
than asking whether a fixed hypergraph is forced by a hereditary density
condition, we ask when a linearly quasirandom hypergraph must contain a
spanning tight cycle.  This combines the Erd{\H{o}}s--S{\'o}s notion of uniform density with
Dirac-type conditions for Hamiltonicity. The prototype is Dirac's theorem~\cite{Dirac1952}, which states that every
graph on $n\ge3$ vertices with minimum degree at least $n/2$ contains
a Hamilton cycle.

We focus on 3-uniform hypergraphs.  A \emph{3-graph} is a pair
$H=(V,E)$ with $E\subseteq\binom V3$.  A tight path is an ordering
$v_1,\ldots,v_m$ of distinct vertices whose edges are the consecutive
triples $v_iv_{i+1}v_{i+2}$ for $i=1,\ldots,m-2$.  Imposing the same
condition cyclically gives a tight cycle, and a tight Hamilton cycle is
a tight cycle containing every vertex of $H$. 
More generally, for $1\le\ell<k$, an $\ell$-cycle in a $k$-graph is
obtained from a cyclic ordering of its vertices by taking blocks of $k$
consecutive vertices whose initial positions are spaced $k-\ell$
vertices apart.  Thus, consecutive edges overlap in exactly $\ell$
vertices.  An $\ell$-cycle containing every vertex is called
Hamiltonian.  The cases $\ell=1$ and $\ell=k-1$ are referred to as
loose and tight cycles, respectively.

Hamiltonicity under linear quasirandomness was first established for
loose cycles.  Lenz, Mubayi and Mycroft~\cite{LenzMubayiMycroft2016} proved that $(p,\mu)$-density,
together with a mild minimum vertex-degree condition, forces a Hamilton
$1$-cycle.  Han, Shu and Wang~\cite{HanShuWangSODA2021,HanShuWang2026} subsequently
developed an absorption framework for non-linear Hamilton cycles in
linearly quasirandom and uniformly dense hypergraphs.  For tight Hamilton cycles
in $3$-graphs, previous positive results required stronger
quasirandomness assumptions, including cherry-quasirandomness and
localised codegree-density conditions
\cite{AignerHorevLevy2021,GanHan2022,
AraujoPigaSchacht2022}.

The local condition considered in this paper is minimum codegree. For a pair $xy\in\binom{V(H)}2$, let $N_H(x,y)=\{z\in V(H):xyz\in E(H)\}$ and $d_H(x,y)=|N_H(x,y)|$.  The minimum codegree of $H$ is
\[
    \delta_2(H)=\min_{xy\in\binom{V(H)}2}d_H(x,y).
\]

Dirac-type conditions for Hamilton cycles in dense uniform hypergraphs
have been studied extensively. Katona and Kierstead introduced
Hamilton $\ell$-cycles in uniform hypergraphs
\cite{KatonaKierstead1999}.  R{\"o}dl, Ruci{\'n}ski and
Szemer{\'e}di proved the asymptotic minimum-codegree threshold for tight
Hamilton cycles~\cite{RodlRucinskiSzemeredi2008}.  K{\"u}hn, Mycroft
and Osthus treated the complementary case $(k-\ell)\nmid k$; together,
these results asymptotically determine the minimum-codegree threshold
for every Hamilton $\ell$-cycle~\cite{KuhnMycroftOsthus2010}. 
For minimum vertex degree in $3$-graphs, Reiher, R{\"o}dl,
Ruci{\'n}ski, Schacht and Szemer{\'e}di proved that the asymptotically
sharp threshold is $5/9$
\cite{ReiherRodlRucinskiSchachtSzemeredi2019}.  The aim of the
present paper is to determine the minimum-codegree threshold when the
host $3$-graph is assumed to be linearly quasirandom.

\subsection{Our results}

Ara{\'u}jo, Piga and Schacht studied tight Hamilton cycles under
localised codegree-density assumptions and exhibited a balanced
$(1/4,1/4)$ obstruction~\cite[Example~1.2]{AraujoPigaSchacht2022}.
This suggested that, under the weaker $(p,\mu)$-density condition, the
simultaneous assumptions $p>1/4$ and $\delta_2(H)>(1/4)n$ might already
force a tight Hamilton cycle.  The question was stated explicitly in
the SODA version of Han, Shu and Wang as Conjecture~1.10
\cite{HanShuWangSODA2021}.

\begin{problem}\cite{AraujoPigaSchacht2022,HanShuWangSODA2021}
\label{pro1}
Fix $p,\alpha>1/4$.  Do there exist $n_0$ and $\mu>0$ such that every
$(p,\mu)$-dense $3$-graph $H$ on $n\ge n_0$ vertices with
$\delta_2(H)\ge\alpha n$ contains a tight Hamilton cycle?
\end{problem}

We answer this question as part of a complete determination of the
minimum-codegree threshold over the full density range $p\in(0,1)$.

Put
\[
    p_\star:=1-\frac1{\sqrt2}
        =0.292893\ldots,
\]
and let $p_0$ be the unique solution of $p_0=(1-p_0)^3$ in $(1/4,1/3)$.  Thus
\[
    p_0=0.3176721961\ldots.
\]
For $p\in(0,1)$, put $r:=1-p$ and define
\begin{equation}
\label{eq:full-threshold}
 \thr(p):=
 \begin{cases}
  \dfrac13,
      &0<p\le p_\star,\\[1.2ex]
  \dfrac{r+4r^2-2r^3}{6},
      &p_\star<p\le p_0,\\[1.2ex]
  \dfrac{r^2}{2},
      &p_0<p\le\dfrac13,\\[1.2ex]
  \left(\dfrac{1-\sqrt{(4p-1)/3}}2\right)^2,
      &\dfrac13<p<1.
 \end{cases}
\end{equation}

\begin{theorem}[Complete threshold theorem]
\label{thm:main}
For every $p\in(0,1)$ and every $\alpha>\thr(p)$, there exist
$\mu>0$ and $n_0$ such that every $(p,\mu)$-dense $3$-graph $H$ on
$n\ge n_0$ vertices with 
\[
    \delta_2(H)\ge\alpha n
\]
contains a tight Hamilton cycle.

Conversely, for every $p\in(0,1)$ and every $\eps,\mu>0$, there exist
arbitrarily large $(p,\mu)$-dense $3$-graphs $H$ satisfying
\[
    \delta_2(H)\ge\bigl(\thr(p)-\eps\bigr)n
\]
and containing no tight Hamilton cycle.  Thus $\thr(p)$ is the
asymptotically sharp minimum-codegree threshold at density $p$.
\end{theorem}

For convenience in the proofs, we refer to the four branches of
\eqref{eq:full-threshold} as
\[
\begin{aligned}
    \thr_1(p)&:=\frac13,
        &&0<p\le p_\star,\\
    \thr_2(p)&:=\frac{r+4r^2-2r^3}{6},
        &&p_\star<p\le p_0,\\
    \thr_3(p)&:=\frac{r^2}{2},
        &&p_0<p\le\frac13,\\
    \thr_4(p)&:=
        \left(\frac{1-\sqrt{(4p-1)/3}}2\right)^2,
        &&\frac13<p<1.
\end{aligned}
\]

\begin{figure}[htbp]
\centering
\begin{tikzpicture}[x=1cm,y=25cm,>=Latex]

\definecolor{curveblue}{RGB}{0,0,180}

\def\xps{3}      
\def\xk{6}       
\def\xthird{8}   
\def\xone{12.0}    

\def\yone{0.333333}
\def\ybeta{0.298211}
\def\ygamma{0.222786}
\def\ytwonine{0.177777}
\def\yonenine{0.088888}
\def\yzero{0}

\draw[->,thick] (0,0) -- (13.2,0) node[right] {$p$};
\draw[->,thick] (0,0) -- (0,0.365) node[above] {$\thr(p)$};

\draw (0,0.005) -- (0,-0.005) node[below=8pt] {$0$};

\draw (\xps,0.005) -- (\xps,-0.005);
\node[below=8pt] at (\xps,0) {$p_\star$};
\node[below=22pt] at (\xps,0) {$0.2929$};

\draw (\xk,0.005) -- (\xk,-0.005);
\node[below=8pt] at (\xk,0) {$p_0$};
\node[below=22pt] at (\xk,0) {$0.3177$};

\draw (\xthird,0.005) -- (\xthird,-0.005) node[below=8pt] {$\dfrac13$};
\draw (\xone,0.005) -- (\xone,-0.005) node[below=8pt] {$1$};

\draw (0.12,\yone) -- (-0.05,\yone) node[left=10pt] {$\dfrac13$};
\draw (0.12,\ytwonine) -- (-0.05,\ytwonine) node[left=10pt] {$\dfrac29$};
\draw (0.12,\yonenine) -- (-0.05,\yonenine) node[left=10pt] {$\dfrac19$};
\draw (0.12,0) -- (-0.05,0) node[left=10pt] {$0$};

\draw[dashed,gray] (0,\ybeta) -- (\xk,\ybeta);
\node[left=10pt] at (0.1,\ybeta) {$0.3182$};

\draw[dashed,gray] (0,\ygamma) -- (\xk,\ygamma);
\node[left=10pt] at (0.1,\ygamma) {$0.2328$};

\draw[dashed,gray] (0,\ytwonine) -- (\xthird,\ytwonine);
\draw[dashed,gray] (0,\yonenine) -- (\xthird,\yonenine);

\draw[dashed,gray] (\xps,0) -- (\xps,\yone);
\draw[dashed,gray] (\xk,0) -- (\xk,\ybeta);
\draw[dashed,gray] (\xthird,0) -- (\xthird,\ytwonine);

\draw[curveblue,very thick] (0,\yone) -- (\xps,\yone);

\draw[curveblue,very thick]
    (\xps,\yone)
    .. controls (3.5,0.325) and (4.5,0.312) ..
    (\xk,\ybeta);

\draw[curveblue,very thick]
    (\xk,\ygamma)
    .. controls (6.5,0.205) and (7.6,0.185) ..
    (\xthird,\ytwonine);

\draw[curveblue,very thick]
    (\xthird,\yonenine)
    .. controls (9.0,0.0005) and (11.5,0.000005) ..
    (\xone,0);

\draw[curveblue,fill=white,thick] 
(0,\yone) circle (2.5pt);
\fill[curveblue] (\xps,\yone) circle (2.5pt);
\fill[curveblue] (\xk,\ybeta) circle (2.5pt);
\draw[curveblue,fill=white,thick] (\xk,\ygamma) circle (2.5pt);
\fill[curveblue] (\xthird,\ytwonine) circle (2.5pt);
\draw[curveblue,fill=white,thick] (\xthird,\yonenine) circle (2.5pt);
\draw[curveblue,fill=white,thick] 
(\xone,0) circle (2.5pt);


\node at (6.3,0.334)
{$\dfrac{r+4r^2-2r^3}{6}$};

\node at (7.6,0.226)
{$\dfrac{r^2}{2}$};

\node at (11,0.1)
{$\left(\dfrac{1-\sqrt{(4p-1)/3}}{2}\right)^2$};

\end{tikzpicture}
\caption{The threshold function $\thr(p)$.}
\label{fig:threshold-function}
\end{figure}

\Cref{fig:threshold-function} depicts the threshold function
$\thr(p)$, where $r:=1-p$.  The two interior points marked on the
horizontal axis are $p_\star:=1-1/\sqrt2$ and the unique
$p_0\in(1/4,1/3)$ satisfying $p_0=(1-p_0)^3$.
The two additional values marked on the vertical axis are
$\thr_2(p_0)\approx0.3182$ and
$\thr_3(p_0)\approx0.2328$.
Neither axis is drawn to scale; the positions of the transition
points and the special branch values are adjusted to display the
four regimes clearly.

The first branch of \Cref{thm:main} gives a particularly strong
negative answer to \Cref{pro1}.

\begin{corollary}[The threshold at density $1/4$]
\label{cor:quarter-density}
At density $p=1/4$, the asymptotically sharp minimum-codegree threshold
for a tight Hamilton cycle is $1/3$.

More precisely, for every $\alpha>1/3$, there exist $\mu>0$ and $n_0$
such that every $(1/4,\mu)$-dense $3$-graph $H$ on $n\ge n_0$
vertices with
\[
    \delta_2(H)\ge\alpha n
\]
contains a tight Hamilton cycle.  Conversely, for every $\eps,\mu>0$,
there exist arbitrarily large $(1/4,\mu)$-dense $3$-graphs $H$ with
\[
    \delta_2(H)\ge
    \left(\frac13-\eps\right)n
\]
and no tight Hamilton cycle.
\end{corollary}

Thus, at density $1/4$, the relevant codegree value is $1/3$, rather
than $1/4$.  Moreover, since the first branch remains constant on the
whole interval $(0,p_\star]$, one may choose any $p \in(1/4,p_\star]$ and any
$\alpha\in(1/4,1/3)$ to obtain a direct counterexample to \Cref{pro1}.

\begin{remark}
The first version of this work already gave a negative answer to
\Cref{pro1} using the same two-part red--blue palette as the construction
behind the first branch of \eqref{eq:full-threshold}.  There the
parameters were chosen so that both the density and the normalised
minimum codegree were asymptotic to $p_0$.

The present optimisation reveals the stronger conclusion that, for
every $p\le p_\star$, the same underlying palette yields a codegree
obstruction asymptotic to $1/3$.  Thus the constants $1/4$, $1/3$, and
$p_0$ play different roles: $1/4$ is the proposed simultaneous
density--codegree bound, $1/3$ is the sharp codegree threshold in the
first regime, and $p_0$ is the density at which the middle space
obstruction disappears.
\end{remark}

The function $\thr$ has three genuine transition points.  At
$p=p_\star$, the first two formulae agree.  At $p=p_0$, the feasible
space obstruction disappears and the threshold drops from
$\thr_2(p_0)$ to the right-hand value $\thr_3(p_0)$.  Finally, at
$p=1/3$, the threshold equals $2/9$, whereas the right-hand limit of
the above-one-third branch is $1/9$.  The two discontinuities reflect changes in the available extremal
mechanisms, rather than intersections of the displayed formulae.

\subsection{Proof strategy}

The proof naturally separates at density $1/3$, but both parts use the
same ordered pair-state regularity framework.

\paragraph{The range $p>1/3$.}
We use the absorption framework of Han, Shu and Wang
\cite{HanShuWang2026}.  Roughly speaking, this framework reduces
Hamiltonicity to a robust connecting property.  If every two ordered
endpoint pairs can be joined by many vertex-disjoint tight paths of one
fixed length, both in the whole hypergraph and inside a suitable
reservoir, then the absorbing method produces a tight Hamilton cycle.
The principal technical input is the following fixed-length connecting
lemma.

\begin{lemma}[Connecting lemma]
\label{lem:connecting}
For every $p>1/3$ and every $\alpha>\thr_4(p)$, there exist $\mu>0$,
$n_0$ and an integer $t$ such that the following holds for all
$n\ge n_0$.  If $H$ is an $n$-vertex $(p,\mu)$-dense $3$-graph with
$\delta_2(H)\ge\alpha n$, then any two disjoint ordered pairs can be
joined by a tight path with exactly $t$ internal vertices.
\end{lemma}

The difficulty is that $(p,\mu)$-density tests only vertex subsets,
whereas a tight path evolves through its terminal ordered pair.  A
regular slice partitions ordered cluster pairs into states; a dense
regular triad gives a directed transition.  Above $1/3$ the whole
pair-state digraph is strongly connected and aperiodic.  A finite
scalar lemma excludes every non-trivial closed state family, and the
extension lemma lifts a fixed-length reduced walk to an actual tight
path.  This part of the proof is kept essentially unchanged from the
original version of the paper.

\paragraph{The range $0<p\le1/3$.}
Arbitrary-pair connectability is false below $1/3$, even when the
minimum codegree is slightly larger than $n/3$; the construction in
\Cref{subsec:con3} records this obstruction.  We therefore use the
Hamilton-framework formalism of Lang and Sanhueza-Matamala
\cite{LangSanhuezaMatamala2026}, in the specialised form used by Lin,
Wang and Zhou~\cite{LinWangZhou2026}.  Instead of proving that the
entire state digraph is primitive, we colour every pair by its actual
tight component and select one canonical transition component.  The
finite structural theorem proves that this component is spanning,
has a robust perfect fractional matching, is aperiodic, and is
consistent under two one-root deletions.

The remaining issue is the passage from the reduced component to one
actual tight component.  We use a locally cleaned multicolour regular slice refining the actual component-shadow colours.  This keeps each
state inside one actual colour, handles repeated indices by coherent
tagged copies, and provides the vertex-lower-regularity needed to
average an actual fractional cover.  We then prove that the selected
actual component is independent of every bounded refinement, attach
all exceptional vertices, and use the interior of the reduced matching
cone to obtain an exact perfect fractional matching on the actual
vertex set.

The three branches below $1/3$ arise from two competing obstructions.
A rooted fractional-matching separator gives
\[
    \alpha\le \thr_1(p)+o(1),\qquad
    \alpha\le \thr_2(p)+o(1),\qquad
    p\le(1-p)^3+o(1),
\]
whereas a two-pocket or period-three obstruction gives 
\[
    \alpha\le \thr_3(p)+o(1).
\]
For $p\le p_\star$, the active space bound is
$\thr_1(p)=1/3$.  On the interval $p_\star<p\le p_0$, the space
obstruction remains feasible and its sharp bound is $\thr_2(p)$.
Once $p>p_0$, the condition $p\le(1-p)^3+o(1)$ becomes impossible,
so the space obstruction disappears and the remaining sharp bound is
$\thr_3(p)$.

\paragraph{Paper organisation.}
The lower-bound constructions and the obstruction to arbitrary-pair
connectability are given in Section~\ref{sec:constructions}. Section~\ref{sec:mainthm}
contains the two upper-bound reductions, separating the ranges above and
at or below density $1/3$. The common
regularity toolkit is developed in Section~\ref{sec:regularity}, and the
connecting lemma above $1/3$ is proved in Section~\ref{sec:connecting}.
At and below $1/3$, Section~\ref{sec:canonical-lifting} develops the
local Hamilton framework and the lifting of the canonical component,
whose structural properties are proved in
Section~\ref{sec:canonical-proof}.  We conclude in
Section~\ref{sec:concluding}.

\section{Constructions}
\label{sec:constructions}

In this section we give the lower-bound constructions corresponding to
all four branches of the threshold function in
\eqref{eq:full-threshold}.  The first two are space barriers, the third
is a rooted two-pocket barrier, and the fourth is the
biased Ara{\'u}jo--Piga--Schacht construction.  We finish with a
separate construction showing that, below density $1/3$, Hamiltonicity
and arbitrary-pair connectability have genuinely different codegree
thresholds.

All the constructions are obtained from quasirandom colourings of
pairs.  We use the following standard concentration statement.

\begin{lemma}[Finite-palette concentration]
\label{lem:palette-concentration}
Fix a finite vertex-partition template and, for every unordered pair of
parts, a probability distribution on a finite set of pair colours.
Colour all pairs independently according to these distributions, with
deterministic colours allowed.  Let $H$ be the $3$-graph obtained by
retaining a fixed collection of coloured triangle patterns.  Then,
with probability $1-o(1)$, the following statements hold.
\begin{enumerate}[label=\textup{(\roman*)},leftmargin=2.5em]
    \item Uniformly over all $U_1,U_2,U_3\subseteq V(H)$, the number
    of retained ordered triples in $U_1\times U_2\times U_3$ differs
    from its expectation by $o(n^3)$.

    \item Uniformly over all vertex pairs of every fixed part type and
    colour, the codegree differs from its conditional expectation by
    $o(n)$.
\end{enumerate}
\end{lemma}

\begin{proof}
For fixed $U_1,U_2,U_3$, the retained-triple count is a function of
$O(n^2)$ independent pair colours, and changing one pair colour changes
this count by at most $O(n)$.  The bounded-differences inequality
therefore gives failure probability $\exp(-\Omega(n^2))$ for an error
of order $\varepsilon n^3$.  A union bound over the at most $8^n$
ordered triples of vertex subsets proves \textup{(i)}.

After conditioning on the colour of a fixed pair, its codegree is a sum
of $O(n)$ independent bounded random variables, one for each possible
third vertex.  Chernoff's inequality and a union bound over the
$O(n^2)$ pairs prove \textup{(ii)}.
\end{proof}

\subsection{The first branch: a two-part space barrier}
\label{subsec:first-branch-construction}

\begin{construction}[Two-part space barrier]
\label{con:two-part}
Fix $p\in(0,p_\star]$, put $r:=1-p$, and choose a sufficiently small
constant $\tau>0$.  Partition
\[
    V=X\dot\cup Y
\]
with
\[
    |X|=\left(\frac23+\tau+o(1)\right)n,
    \qquad
    |Y|=\left(\frac13-\tau+o(1)\right)n.
\]
Colour every pair inside $X$ independently red with probability $r$
and blue with probability $p$.  Define a $3$-graph $H$ by retaining
\begin{enumerate}[label=\textup{(\alph*)},leftmargin=2.6em]
    \item every triple of type $YYY$ or $XYY$;
    \item a triple of type $XXY$ if and only if its pair inside $X$ is
    blue; and
    \item a triple contained in $X$ if and only if its three pairs are
    red.
\end{enumerate}
Choose a colouring satisfying \Cref{lem:palette-concentration}.
\end{construction}

\begin{proposition}
\label{prop:two-part-properties}
For every $p\in(0,p_\star]$ and every $\eps,\mu>0$, the parameter
$\tau$ in \Cref{con:two-part} can be chosen so that, for all
sufficiently large $n$, the resulting $3$-graph is $(p,\mu)$-dense,
satisfies
\[
    \delta_2(H)\ge\left(\frac13-\eps\right)n,
\]
and contains no tight Hamilton cycle.
\end{proposition}

\begin{proof}
The densities of the four vertex-part patterns $YYY$, $XYY$, $XXY$
and $XXX$ are, respectively,
\[
    1,\qquad 1,\qquad p,\qquad r^3.
\]
Since $p\le p_\star<p_0$ and $p_0$ is the unique solution of
$p=(1-p)^3$, we have $p\le r^3$.  Thus every part pattern has density
at least $p$, and \Cref{lem:palette-concentration}\textup{(i)} gives
$(p,\mu)$-density for all sufficiently large $n$.

Write $x:=|X|/n$ and $y:=|Y|/n$.  Up to a uniform $o(1)$ error, the
normalised codegrees are
\begin{center}
\begin{tabular}{@{}ll@{}}
\toprule
pair type & normalised codegree \\
\midrule
red pair inside $X$ & $r^2x$ \\
blue pair inside $X$ & $y$ \\
pair between $X$ and $Y$ & $y+px$ \\
pair inside $Y$ & $1$ \\
\bottomrule
\end{tabular}
\end{center}
Indeed, a red pair inside $X$ can be completed only inside $X$, where
the two new pairs must also be red.  A blue pair inside $X$ can be
completed by every vertex of $Y$, while an $XY$-pair can be completed
by every other vertex of $Y$ and by the blue neighbours of its
$X$-endpoint.  Since $r^2\ge1/2$ for $p\le p_\star$, we have
\[
    r^2x\ge\frac13+\frac\tau2+o(1),
    \qquad
    y=\frac13-\tau+o(1).
\]
Choose $\tau<\eps/3$ and then $n$ sufficiently large.  This gives the
claimed minimum-codegree bound.

Suppose that $H$ contains a tight Hamilton cycle.  If one consecutive
pair on the cycle is red, then each adjacent edge must be a red
triangle contained in $X$.  The red colour therefore propagates around
the cycle and forces every vertex of the cycle to lie in $X$, a
contradiction.  Hence the cycle has no red consecutive pair.  In
particular, it contains no three consecutive vertices of $X$, because
an edge contained in $X$ is a red triangle.  Every cyclic block of
vertices from $X$ consequently has length at most two, so
\[
    |X|\le2|Y|,
\]
contrary to $|X|=(2/3+\tau+o(1))n>2|Y|$.
\end{proof}

\subsection{The second branch: a three-layer space barrier}
\label{subsec:middle-layer-construction}

\begin{construction}[Three-layer space barrier]
\label{con:middle-layer}
Fix $p\in(p_\star,p_0]$, put $r:=1-p$, and define
\[
    m_*:=\frac{1-2r^2}{6r},
    \qquad
    y_*:=\frac{2r^2+2r-1}{6r}.
\]
Then $m_*,y_*\ge0$ and $m_*+y_*=1/3$.  Choose a sufficiently small
constant $\tau>0$ and partition
\[
    V=A\dot\cup M\dot\cup Y
\]
so that
\[
    |A|=\left(\frac23+\tau+o(1)\right)n,
    \qquad
    |M|=(m_*+o(1))n,
    \qquad
    |Y|=(y_*-\tau+o(1))n.
\]
Put $X:=A\dot\cup M$.  Colour every pair inside $X$ independently red
with probability $r$ and blue with probability $p$, and colour every
pair meeting $Y$ blue.  Define a $3$-graph $H$ by retaining
\begin{enumerate}[label=\textup{(\alph*)},leftmargin=2.6em]
    \item every red triangle contained in $X$; and
    \item every blue triangle which is not contained entirely in $A$.
\end{enumerate}
Choose a colouring satisfying \Cref{lem:palette-concentration}.
\end{construction}

\begin{proposition}
\label{prop:middle-layer-properties}
For every $p\in(p_\star,p_0]$ and every $\eps,\mu>0$, the parameter
$\tau$ in \Cref{con:middle-layer} can be chosen so that, for all
sufficiently large $n$, the resulting $3$-graph is $(p,\mu)$-dense,
satisfies
\[
    \delta_2(H)\ge\bigl(\thr_2(p)-\eps\bigr)n,
\]
and contains no tight Hamilton cycle.
\end{proposition}

\begin{proof}
A triple contained in $A$ is retained when its three pairs are red, and
therefore has density $r^3$.  Since $p\le p_0$, one has $p\le r^3$.
A triple contained in $X$ and meeting $M$ is retained when it is all
red or all blue, and hence has density $r^3+p^3\ge p$.  A triple with
exactly one vertex in $Y$ and two vertices in $X$ is retained precisely
when the pair inside $X$ is blue, so its density is $p$.  Every triple
with at least two vertices in $Y$ is retained.  Thus every part pattern
has density at least $p$, and \Cref{lem:palette-concentration}\textup{(i)}
gives $(p,\mu)$-density.

Write
\[
    \ell:=\frac{|A|}{n},
    \qquad
    m:=\frac{|M|}{n},
    \qquad
    y:=\frac{|Y|}{n},
    \qquad
    x:=\ell+m=1-y.
\]
The asymptotic codegrees are
\begin{center}
\begin{tabular}{@{}ll@{}}
\toprule
pair type & normalised codegree \\
\midrule
red pair in $X$ & $r^2x$ \\
blue pair in $A$ & $y+p^2m$ \\
blue pair in $X$ with an endpoint in $M$ & $y+p^2x$ \\
pair between $X$ and $Y$ & $y+px$ \\
pair inside $Y$ & $1$ \\
\bottomrule
\end{tabular}
\end{center}
At $\tau=0$, direct substitution gives
\[
    r^2(1-y_*)
    =y_*+p^2m_*
    =\frac{r+4r^2-2r^3}{6}
    =\thr_2(p).
\]
The other three rows are strictly larger.  Choosing $\tau$ sufficiently
small and then $n$ sufficiently large gives
\[
    \delta_2(H)\ge\bigl(\thr_2(p)-\eps\bigr)n.
\]

Every edge of $H$ is a monochromatic triangle.  Since tightly adjacent
edges share a pair, every tight cycle is monochromatic.  A red tight
cycle is contained in $X$ and therefore misses $Y$.  Every blue edge,
on the other hand, meets
\[
    S:=M\cup Y,
\]
and
\[
    |S|=(m_*+y_*-\tau+o(1))n
       =\left(\frac13-\tau+o(1)\right)n.
\]
A tight Hamilton cycle has exactly $n$ edges and each vertex lies in
exactly three of them.  If such a cycle were blue, counting incidences
between its edges and $S$ would give
\[
    n\le3|S|<n,
\]
a contradiction.
\end{proof}

\begin{remark}
At $p=p_\star$, one has $m_*=0$, and the limiting form of
\Cref{con:middle-layer} is the two-part construction.  The middle class
appears precisely when $p>p_\star$.
\end{remark}

\subsection{The third branch: a rooted two-pocket barrier}
\label{subsec:rooted-two-pocket-construction}

\begin{construction}[Rooted two-pocket barrier]
\label{con:rooted-two-pocket}
Fix $p\in(p_0,1/3]$ and put $r:=1-p$.  Let
\[
    V=W\dot\cup\{x_{\mathrm R},x_{\mathrm B}\},
    \qquad
    W=X\dot\cup Y,
\]
where $|X|,|Y|=(1/2+o(1))n$.  Colour pairs inside $X$ and inside $Y$
independently red with probability $r$ and blue with probability $p$,
and colour every pair between $X$ and $Y$ blue.  Define $H$ as follows.
\begin{enumerate}[label=\textup{(\alph*)},leftmargin=2.6em]
    \item A triple contained in $W$ is an edge if and only if it is a
    monochromatic triangle.

    \item For $u,v\in W$, the triple $x_{\mathrm R}uv$ is an edge if
    and only if $uv$ is red, while $x_{\mathrm B}uv$ is an edge if and
    only if $uv$ is blue.

    \item Every triple $x_{\mathrm R}x_{\mathrm B}v$ with $v\in W$ is
    an edge.
\end{enumerate}
Choose a colouring satisfying \Cref{lem:palette-concentration}.
\end{construction}

\begin{proposition}
\label{prop:rooted-two-pocket-properties}
For every $p\in(p_0,1/3]$ and every $\eps,\mu>0$, all sufficiently
large instances of \Cref{con:rooted-two-pocket} are $(p,\mu)$-dense,
satisfy
\[
    \delta_2(H)\ge\bigl(\thr_3(p)-\eps\bigr)n,
\]
and contain no tight Hamilton cycle.
\end{proposition}

\begin{proof}
The two roots affect every vertex-set density count by only $O(n^2)$,
so it is enough to consider triples in $W$.  A triple contained in one
of $X,Y$ is retained with density $r^3+p^3$, and
\[
    r^3+p^3-p=(1-p)(1-3p)\ge0
\]
for $p\le1/3$.  A triple meeting both $X$ and $Y$ is retained precisely
when its same-part pair is blue, and therefore has density $p$.
Hence \Cref{lem:palette-concentration}\textup{(i)} gives
$(p,\mu)$-density.

The asymptotic normalised codegrees are
\begin{center}
\begin{tabular}{@{}ll@{}}
\toprule
pair type & normalised codegree \\
\midrule
red pair inside $X$ or $Y$ & $r^2/2$ \\
blue pair inside $X$ or $Y$ & $(1+p^2)/2$ \\
blue pair between $X$ and $Y$ & $p$ \\
$x_{\mathrm R}v$, $v\in W$ & $r/2$ \\
$x_{\mathrm B}v$, $v\in W$ & $(1+p)/2$ \\
$x_{\mathrm R}x_{\mathrm B}$ & $1$ \\
\bottomrule
\end{tabular}
\end{center}
with a uniform $o(1)$ error.  For $p>p_0$, every row is at least
$r^2/2-o(1)$.  In particular,
\[
    \delta_2(H)\ge\left(\frac{r^2}{2}-\eps\right)n
                 =\bigl(\thr_3(p)-\eps\bigr)n.
\]

Suppose that $H$ contains a tight Hamilton cycle.  The two roots split
it into two arcs, one of which has the form
\[
    x_{\mathrm R},v_1,v_2,\ldots,v_t,x_{\mathrm B}
\]
with $t\ge2$ and $v_1,\ldots,v_t\in W$.  The first edge forces
$v_1v_2$ to be red.  Every internal edge is a monochromatic triangle,
so all consecutive pairs $v_iv_{i+1}$ along the arc are red.  The last
edge, however, forces $v_{t-1}v_t$ to be blue, a contradiction.
\end{proof}

\subsection{The fourth branch: the biased Ara{\'u}jo--Piga--Schacht barrier}
\label{subsec:above-third-construction}

\begin{construction}[Biased propagation barrier]
\label{con:sharp-curve}
Fix $p\in(1/3,1)$ and set
\[
    \rho:=\frac{1-\sqrt{(4p-1)/3}}{2}.
\]
Then $\rho\in(0,1/3)$,
\[
    p=\rho^3+(1-\rho)^3,
    \qquad
    \rho^2=\thr_4(p).
\]
Let $W$ be a set of $n-2$ vertices and let $x,y$ be two additional
vertices.  Colour every pair of $W$ independently red with probability
$\rho$ and blue with probability $1-\rho$.  Define a $3$-graph $H$ on
$W\cup\{x,y\}$ as follows.
\begin{enumerate}[label=\textup{(\alph*)},leftmargin=2.6em]
    \item A triple contained in $W$ is an edge if and only if it is
    monochromatic.

    \item For $u,v\in W$, the triple $xuv$ is an edge if and only if
    $uv$ is red, while $yuv$ is an edge if and only if $uv$ is blue.

    \item Every triple $xyv$ with $v\in W$ is an edge.
\end{enumerate}
Choose a colouring satisfying \Cref{lem:palette-concentration}.
\end{construction}

\begin{proposition}
\label{prop:sharp-curve}
For every $p\in(1/3,1)$ and every $\eps,\mu>0$, all sufficiently large
instances of \Cref{con:sharp-curve} are $(p,\mu)$-dense, satisfy
\[
    \delta_2(H)\ge\bigl(\thr_4(p)-\eps\bigr)n,
\]
and contain no tight Hamilton cycle.
\end{proposition}

\begin{proof}
The two special vertices affect only $O(n^2)$ ordered triples.  A
triple contained in $W$ is retained with probability
\[
    \rho^3+(1-\rho)^3=p,
\]
so \Cref{lem:palette-concentration}\textup{(i)} gives
$(p,\mu)$-density.

If $uv\subseteq W$ is red, then its extensions consist of $x$ and the
common red neighbours of $u,v$ in $W$, giving
\[
    d_H(u,v)=(\rho^2+o(1))n.
\]
If $uv$ is blue, the analogous value is
$((1-\rho)^2+o(1))n$, which is larger because $\rho<1/2$.  Moreover,
for every $u\in W$,
\[
    d_H(x,u)=(\rho+o(1))n,
    \qquad
    d_H(y,u)=((1-\rho)+o(1))n,
\]
and $d_H(x,y)=n-2$.  Hence
\[
    \delta_2(H)\ge(\rho^2-\eps)n
                 =\bigl(\thr_4(p)-\eps\bigr)n.
\]

Suppose that $H$ contains a tight Hamilton cycle.  Of the two arcs
between $x$ and $y$, choose one with at least two internal vertices and
write it as
\[
    x,w_1,w_2,\ldots,w_m,y,
    \qquad m\ge2.
\]
The first edge forces $w_1w_2$ to be red.  Since every internal edge is
monochromatic, the red colour propagates along the arc, so
$w_{m-1}w_m$ is red.  The last edge forces the same pair to be blue, a
contradiction.
\end{proof}

\begin{corollary}[Complete lower bound]
\label{cor:complete-lower-bound}
For every $p\in(0,1)$ and every $\eps,\mu>0$, there are arbitrarily
large $(p,\mu)$-dense $3$-graphs with minimum codegree at least
\[
    \bigl(\thr(p)-\eps\bigr)n
\]
and no tight Hamilton cycle.
\end{corollary}

\begin{proof}
Use \Cref{prop:two-part-properties} when $p\le p_\star$,
\Cref{prop:middle-layer-properties} when $p_\star<p\le p_0$,
\Cref{prop:rooted-two-pocket-properties} when $p_0<p\le1/3$, and
\Cref{prop:sharp-curve} when $p>1/3$.
\end{proof}

\subsection{Hamiltonicity versus arbitrary-pair connectability below one third}
\label{subsec:con3}

We finally record why the global connecting lemma used above density
$1/3$ cannot be the correct tool below this barrier.  Call a $3$-graph
\emph{globally pair-connectable} if every two disjoint ordered pairs
can be joined by a tight path.

\begin{construction}[Obstruction to arbitrary-pair connectability]
\label{con:pair-connectability-obstruction}
Fix $\rho\in(2/3,1)$ and partition
\[
    V=X\dot\cup Y,
    \qquad
    |X|=\left(\frac{1}{2\rho}+o(1)\right)n,
    \qquad
    |Y|=n-|X|.
\]
Colour every pair inside $X$ independently red with probability $\rho$
and blue with probability $1-\rho$.  Define $H$ by retaining
\begin{enumerate}[label=\textup{(\alph*)},leftmargin=2.6em]
    \item every triple of type $YYY$ or $XYY$;
    \item a triple of type $XXY$ if and only if its pair inside $X$ is
    blue; and
    \item a triple contained in $X$ if and only if it is a
    monochromatic triangle.
\end{enumerate}
Choose a colouring satisfying \Cref{lem:palette-concentration}.
\end{construction}

\begin{proposition}
\label{prop:pair-connectability-obstruction}
For every fixed $\rho\in(2/3,1)$ and every $\eps,\mu>0$, all
sufficiently large instances of
\Cref{con:pair-connectability-obstruction} are
$(1-\rho,\mu)$-dense and satisfy
\[
    \delta_2(H)\ge\left(\frac\rho2-\eps\right)n.
\]
Nevertheless, some two disjoint ordered pairs cannot be joined by any
tight path.
\end{proposition}

\begin{proof}
The densities of the four part patterns $YYY$, $XYY$, $XXY$ and $XXX$
are
\[
    1,\qquad 1,\qquad 1-\rho,
    \qquad \rho^3+(1-\rho)^3.
\]
Since $\rho>2/3$,
\[
    \rho^3+(1-\rho)^3-(1-\rho)
    =\rho(3\rho-2)>0.
\]
Thus the minimum part-pattern density is $1-\rho$, and the concentration
lemma gives $(1-\rho,\mu)$-density.

Up to a uniform $o(1)$ error, the normalised codegrees are
\begin{center}
\begin{tabular}{@{}ll@{}}
\toprule
pair type & normalised codegree \\
\midrule
pair inside $Y$ & $1$ \\
pair between $X$ and $Y$ & $1/2$ \\
red pair inside $X$ & $\rho/2$ \\
blue pair inside $X$ & $\rho/2$ \\
\bottomrule
\end{tabular}
\end{center}
For example, a red pair inside $X$ has
$\rho^2|X|=(\rho/2+o(1))n$ common red neighbours in $X$, while a blue
pair inside $X$ has every vertex of $Y$ and
$(1-\rho)^2|X|+o(n)$ common blue neighbours, whose total is again
$(\rho/2+o(1))n$.  This proves the codegree bound.

Let $(u,v)$ be an ordered red pair inside $X$.  If a tight path starts
with $(u,v)$ and has next vertex $w$, then $uvw$ being an edge forces
$w\in X$, because a red $X$-pair has no extension in $Y$.  Moreover,
$uvw$ is a monochromatic triangle, so the new terminal pair $(v,w)$ is
again red.  Iterating shows that every vertex of the path lies in $X$
and every consecutive pair is red.  Hence $(u,v)$ cannot be joined to
any disjoint ordered pair containing a vertex of $Y$.
\end{proof}

Taking $\rho=2/3+\xi$, where
\[
    0<\xi<\frac13-p_0,
\]
gives density
\[
    p=1-\rho=\frac13-\xi>p_0
\]
and normalised minimum codegree asymptotic to
\[
    \frac\rho2=\frac13+\frac\xi2.
\]
At the same density, the Hamiltonicity threshold is
\[
    \thr_3(p)=\frac{(1-p)^2}{2}=\frac{\rho^2}{2}
             <\frac\rho2.
\]
Consequently, after choosing the quasirandomness error sufficiently
small, \Cref{thm:main} implies that these examples contain a tight
Hamilton cycle, even though they are not globally pair-connectable.
Thus, below density $1/3$, the codegree condition forcing arbitrary-pair
connectability is strictly stronger than the condition forcing a tight
Hamilton cycle.  This final observation is not used in the proof of
\Cref{thm:main}.

\section{The two upper-bound reductions}
\label{sec:mainthm}

The proof of the upper bound in \Cref{thm:main} separates naturally at
density $1/3$.  Above $1/3$, the arbitrary-pair connecting lemma is
available, and the result follows from the absorption framework of Han,
Shu and Wang.  At and below $1/3$, arbitrary-pair connectability is no
longer forced at the Hamiltonicity threshold, so we instead reduce the
problem to a local Hamilton-framework theorem.  We carry out these two
reductions separately.

\subsection{The range above one third}
\label{subsec:above-third-reduction}

Throughout this subsection, $p>1/3$, so the relevant branch of the
threshold function is $\thr_4(p)$.  We first recall the absorption
statement used in the original proof.

Let $1\le \ell<k$, and let $\gamma>0$ and $t\in\mathbb N$.  A $k$-graph
$H$ is called \emph{$(k,\ell,\gamma,t)$-connectable} if, for every two
disjoint ordered $\ell$-sets $L_1,L_2\subseteq V(H)$, there are at least
$\gamma|V(H)|$ pairwise vertex-disjoint sets $T\subseteq
V(H)\setminus(L_1\cup L_2)$ of size $t$ such that the concatenation
$L_1TL_2$ forms an $\ell$-path.

\begin{theorem}[Absorption framework, {\cite[Theorem~1.9]{HanShuWang2026}}]
\label{thm:absorption-framework}
Given integers $1\le \ell<k$ and $t\in\mathbb N$, and reals
$p,\beta>0$, there exist $\xi',\xi,\gamma,\mu_0>0$ and an integer
$n_0$ such that the following holds for every $\mu\in(0,\mu_0)$ and
every $n\ge n_0$ with $n\in(k-\ell)\mathbb N$.  Suppose that $H$ is an
$n$-vertex $(p,\mu)$-dense $k$-graph satisfying
$\delta_1(H)\ge\beta n^{k-1}$ and that $H$ is
$(k,\ell,\gamma,t)$-connectable.  Suppose further that there is a set
$A\subseteq V(H)$ of size $\xi n$ such that, for every
$B\subseteq V(H)$ with $|B|\le\xi'n$, the induced hypergraph
$H[A\cup B]$ is $(k,\ell,\gamma,t)$-connectable.  Then $H$ contains a
Hamilton $\ell$-cycle.
\end{theorem}

We shall also use the following standard reservoir lemma.

\begin{lemma}[Reservoir lemma, {\cite[Lemma~8.1]{KuhnMycroftOsthus2010}}]
\label{lem:reservoir}
Let $1/n\ll\varepsilon\ll\xi,\alpha<1$.  Let $H$ be a $3$-graph on
$n$ vertices with $\delta_2(H)\ge\alpha n$, and let
$A\subseteq V(H)$ be chosen uniformly among the sets of size $\xi n$.
Then, with probability $1-o(1)$, every pair
$S\in\binom{V(H)}2$ satisfies
\[
    |N_H(S)\cap A|\ge(\alpha-\varepsilon)|A|.
\]
\end{lemma}

We now deduce the upper bound above $1/3$ from the connecting lemma.

\begin{proof}[Proof of the upper bound for $p>1/3$]
Fix $p>1/3$ and $\alpha>\thr_4(p)$.  Choose
\[
    \thr_4(p)<\alpha^-<\alpha^+<\alpha.
\]
Apply \Cref{lem:connecting} with parameters $p$ and $\alpha^-$.  This
gives $\mu_{\rm con}>0$, an integer $n_{\rm con}$ and a fixed number
$t$ of internal vertices.

Apply \Cref{thm:absorption-framework} with $k=3$, $\ell=2$, this value
of $t$, density parameter $p$, and vertex-degree parameter
$\beta=\alpha/3$.  Let
$\xi',\xi,\gamma,\mu_{\rm abs},n_{\rm abs}$ be the resulting constants.
Choose the remaining parameters so that
\[
    1/n\ll\mu\ll\min\{\mu_{\rm abs},\xi^3\mu_{\rm con}\},
    \qquad
    \varepsilon,\xi'\ll\xi,
    \qquad
    \gamma t\ll
    \min\{\alpha-\alpha^-,\alpha^+-\alpha^-\}.
\]
Let $H$ be an $n$-vertex $(p,\mu)$-dense $3$-graph with
$\delta_2(H)\ge\alpha n$.  For every vertex $v$,
\[
    \deg_H(v)
       =\frac12\sum_{u\ne v}d_H(u,v)
       \ge\frac{\alpha n(n-1)}2,
\]
and hence $\delta_1(H)\ge\alpha n^2/3$ for sufficiently large $n$.
Thus the vertex-degree condition in
\Cref{thm:absorption-framework} is satisfied.

We next prove that $H$ is $(3,2,\gamma,t)$-connectable.  Fix two
disjoint ordered pairs $(x_1,x_2)$ and $(y_1,y_2)$.  Suppose that
vertex-disjoint internal sets $T_1,\ldots,T_{i-1}$ have already been
chosen, and put
\[
    H_i:=H\left[V(H)\setminus\bigcup_{j<i}T_j\right].
\]
Before the greedy procedure terminates, fewer than $\gamma tn$ vertices
have been removed.  Therefore every pair in $H_i$ has codegree at least
\[
    \alpha n-\gamma tn\ge\alpha^-|V(H_i)|.
\]
Moreover, induced subgraphs inherit one-sided density after rescaling
the error term.  Since $|V(H_i)|\ge(1-\gamma t)n$, our choice of $\mu$
ensures that $H_i$ is $(p,\mu_i)$-dense for some
$\mu_i\le\mu_{\rm con}$.  The connecting lemma therefore supplies a
new tight path from $(x_1,x_2)$ to $(y_1,y_2)$ with exactly $t$ internal
vertices.  Continuing greedily gives at least $\gamma n$ pairwise
vertex-disjoint connectors.  Hence $H$ is
$(3,2,\gamma,t)$-connectable.

It remains to verify the reservoir condition.  Choose a set
$A\subseteq V(H)$ of size $\xi n$ satisfying the conclusion of
\Cref{lem:reservoir}.  Let $B\subseteq V(H)$ have size at most $\xi'n$
and put $U:=A\cup B$.  The parameter hierarchy gives
\[
    (\alpha-\varepsilon)|A|\ge\alpha^+|U|,
\]
so every pair in $H[U]$ has codegree at least $\alpha^+|U|$.

We show, again greedily, that $H[U]$ is
$(3,2,\gamma,t)$-connectable.  After deleting fewer than
$\gamma t|U|$ internal vertices, let $U'$ be the remaining set.  Every
pair in $H[U']$ has codegree at least
\[
    \alpha^+|U|-\gamma t|U|\ge\alpha^-|U'|.
\]
Also $|U'|\ge\xi n/2$.  Consequently, the density error in the induced
hypergraph $H[U']$ is at most a constant multiple of $\mu/\xi^3$, and
our choice of $\mu$ makes $H[U']$ $(p,\mu')$-dense with
$\mu'\le\mu_{\rm con}$.  The connecting lemma supplies the next
connector, and the greedy argument produces $\gamma|U|$ pairwise
disjoint connectors.  Thus $H[A\cup B]$ is
$(3,2,\gamma,t)$-connectable for every $B$ with $|B|\le\xi'n$.

All assumptions of \Cref{thm:absorption-framework} are now verified.
Since $k-\ell=1$, there is no divisibility restriction, and $H$ contains
a Hamilton $2$-cycle, that is, a tight Hamilton cycle.
\end{proof}

\subsection{The range at and below one third}
\label{subsec:below-third-reduction}

For $0<p\le1/3$, the construction in
\Cref{prop:pair-connectability-obstruction} shows that the Hamiltonicity
threshold does not force arbitrary-pair connectability.  We therefore
replace the global connecting reduction by the Hamilton-framework
formalism of Lang and Sanhueza-Matamala, in the specialised form used
by Lin, Wang and Zhou
\cite{LangSanhuezaMatamala2026,LinWangZhou2026}.

Two edges of a $3$-graph are \emph{tightly adjacent} if they share a
pair.  A \emph{tight component} is an edge-maximal tightly connected
subgraph.  A \emph{perfect fractional matching} in a $3$-graph $G$ is
a function $\omega:E(G)\to[0,1]$ such that
\[
    \sum_{\substack{e\in E(G)\\v\in e}}\omega(e)=1
    \qquad\text{for every }v\in V(G).
\]
A closed tight walk of order $m$ is the image of a homomorphism
$C_m^{(3)}\to G$.

\begin{definition}[Hamilton framework]
\label{def:framework}
Let $s\ge6$ and let $\cP_s$ be a family of $s$-vertex $3$-graphs.  A
choice $\cF(G)\subseteq G$ for every $G\in\cP_s$ is a Hamilton
framework if the following conditions hold.
\begin{enumerate}[label=\textup{(F\arabic*)},leftmargin=2.8em]
 \item $\cF(G)$ is a spanning tight component of $G$.
 \item $\cF(G)$ has a perfect fractional matching.
 \item $\cF(G)$ contains a closed tight walk whose order is congruent
 to $1$ modulo $3$.
 \item If $J-x,J-y\in\cP_s$ for distinct $x,y\in V(J)$, then
 $\cF(J-x)\cup\cF(J-y)$ is tightly connected in $J$.
\end{enumerate}
\end{definition}

We use the following consequence of the hypergraph bandwidth theorem.

\begin{theorem}[Framework embedding,
{\cite{LangSanhuezaMatamala2026,LinWangZhou2026}}]
\label{thm:framework-embedding}
Fix $s\ge6$ and suppose that $\cP_s$ admits a Hamilton framework.
There exists $n_0$ such that the following holds for every $n\ge n_0$.
Let $H$ be an $n$-vertex $3$-graph.  If, for every fixed $6$-set $R\subseteq V(H)$, at least a
$1-1/s^2$ fraction of the $s$-sets $S\supseteq R$ induce members
of $\cP_s$, then $H$ contains a tight Hamilton
cycle.
\end{theorem}

The one-sided density condition and minimum codegree are inherited by
almost every induced subgraph of sufficiently large fixed order.  For
the density condition, we use \cite[Lemma~8.1]{Lang2023}, while the
codegree inheritance follows from
\cite[Lemma~4.3]{LangSanhuezaMatamala2026}. Applying these two statements to fixed $6$-sets and to codegree, and
absorbing the negligible normalisation error into $\nu$, gives the
following.

\begin{lemma}[Inheritance]
\label{lem:inheritance}
For every $p,\alpha,\nu>0$ and every sufficiently large fixed $s$,
there is $\mu>0$ such that the following holds.  If $H$ is sufficiently
large, $(p,\mu)$-dense and satisfies
$\delta_2(H)\ge\alpha|V(H)|$, then, for every fixed $6$-set
$R\subseteq V(H)$, at least a $1-1/s^2$ fraction of the $s$-sets
$S\supseteq R$ satisfy
\[
    H[S]\text{ is $(p,\nu)$-dense}
    \qquad\text{and}\qquad
    \delta_2(H[S])\ge(\alpha-\nu)s.
\]
\end{lemma}

The finite statement to be proved in the remainder of the paper is the
following.

\begin{theorem}[Local framework theorem]
\label{thm:local-framework}
Fix $p\in(0,1/3]$ and $\alpha>\thr(p)$.  For all sufficiently small
$\nu>0$ and all sufficiently large fixed $s$, the family of
$s$-vertex $(p,\nu)$-dense $3$-graphs with minimum codegree at least
$(\alpha-\nu)s$ admits a Hamilton framework.
\end{theorem}

We now deduce the global upper bound from this local theorem.

\begin{proof}[Proof of the upper bound for $0<p\le1/3$]
Fix $p\in(0,1/3]$ and $\alpha>\thr(p)$.  Choose
\[
    \thr(p)<\bar\alpha<\alpha.
\]
Apply \Cref{thm:local-framework} with parameters $p$ and
$\bar\alpha$.  Choose $\nu>0$ sufficiently small and then a
sufficiently large fixed $s$ so that the theorem applies, and let
$\cP_s$ be the family of $s$-vertex $(p,\nu)$-dense $3$-graphs with
minimum codegree at least $(\bar\alpha-\nu)s$.

Apply \Cref{lem:inheritance} with the global codegree
parameter $\alpha$, decreasing $\mu$ if necessary.  If $H$ is a
sufficiently large $(p,\mu)$-dense $3$-graph with
$\delta_2(H)\ge\alpha|V(H)|$, then, for every fixed $6$-set $R\subseteq V(H)$, at least a
$1-1/s^2$ fraction of the $s$-sets $S\supseteq R$ satisfy
\[
    H[S]\text{ is }(p,\nu)\text{-dense}
    \quad\text{and}\quad
    \delta_2(H[S])\ge(\alpha-\nu)s
                         \ge(\bar\alpha-\nu)s.
\]
Thus all these induced subgraphs belong to $\cP_s$.  By
\Cref{thm:local-framework}, the family $\cP_s$ admits a Hamilton
framework, and \Cref{thm:framework-embedding} gives a tight Hamilton
cycle in $H$.
\end{proof}

Together with the four constructions in \Cref{sec:constructions}, the
two deductions above prove \Cref{thm:main} once
\Cref{lem:connecting,thm:local-framework} have been established.

\section{Regularity tools}
\label{sec:regularity}

The proofs on the two sides of the density barrier $1/3$ use the same
regularity framework, although the reduced structures arising from it
are used in different ways.  In this section we collect the regularity
tools needed throughout the paper and formulate them in a form suitable
for both parts of the argument.

We begin with regular complexes, ordered pair-cells and a locally
cleaned regular slice, and then establish the resulting counting,
extension and codegree-inheritance properties.  Finally, we describe
the additional component-aware refinement needed in the range at and
below $1/3$.

\subsection{Regular complexes}
\label{subsec:regular-complexes}

We use the language of regular complexes.  A $(k,h)$-complex
$\mathcal H$ with vertex classes $V_1,\ldots,V_h$ is an $h$-partite
hypergraph whose edges have size at most $k$ and which is downward
closed: if $e\in E(\mathcal H)$ and $e'\subseteq e$, then
$e'\in E(\mathcal H)$.  We write $\mathcal H^{(j)}$ for its
$j$-uniform level.  If $\mathcal J$ is a $(j-1)$-graph, then
$K_j(\mathcal J)$ denotes the family of $j$-sets whose
$(j-1)$-subsets are all edges of $\mathcal J$.

Let $\mathcal H^{(j)}$ be a $j$-graph supported on a
$(j-1)$-graph $\mathcal H^{(j-1)}$.  We say that
$\mathcal H^{(j)}$ is $(d_j,\delta,r)$-regular with respect to
$\mathcal H^{(j-1)}$ if, for every collection
$\mathcal Q_1,\ldots,\mathcal Q_s$ of at most $r$ subgraphs of
$\mathcal H^{(j-1)}$ satisfying
\[
    \left|\bigcup_{q=1}^s K_j(\mathcal Q_q)\right|
    \ge
    \delta |K_j(\mathcal H^{(j-1)})|,
\]
we have
\[
    \left|
        E(\mathcal H^{(j)})
        \cap
        \bigcup_{q=1}^sK_j(\mathcal Q_q)
    \right|
    =
    (d_j\pm\delta)
    \left|\bigcup_{q=1}^sK_j(\mathcal Q_q)\right|.
\]

Let $\mathbf d=(d_2,\ldots,d_k)$.  A $(k,h)$-complex
$\mathcal H$ is called
\emph{$(\mathbf d,\delta_k,\delta,r)$-regular} if, for every
$2\le j\le k-1$, the level $\mathcal H^{(j)}$ is
$(d_j,\delta,1)$-regular with respect to
$\mathcal H^{(j-1)}$, and the top level $\mathcal H^{(k)}$ is
$(d_k,\delta_k,r)$-regular with respect to
$\mathcal H^{(k-1)}$.  When several cells occur at the same level,
the same definition is applied to each cell separately, and their
densities may be different.  We shall only apply the counting and
extension lemmas to complexes of bounded order whose non-zero densities
are bounded away from zero.

We shall also use the following notation for labelled copies.  Let
$\mathcal G$ and $\mathcal H$ be complexes with respective vertex
classes $X_1,\ldots,X_h$ and $V_1,\ldots,V_h$.  We say that
$\mathcal H$ \emph{respects the partition} of $\mathcal G$ if every
edge of $\mathcal G$ is supported on the corresponding vertex classes
of $\mathcal H$.  Assuming this, a labelled copy of $\mathcal G$ in
$\mathcal H$ is \emph{partition-respecting} if it is given by an
injective map
\[
    \phi:V(\mathcal G)\longrightarrow V(\mathcal H)
\]
such that $\phi(X_i)\subseteq V_i$ for every $i\in[h]$ and
$\phi(e)\in E(\mathcal H)$ for every $e\in E(\mathcal G)$.  We write
$|\mathcal G|_{\mathcal H}$ for the number of labelled
partition-respecting copies of $\mathcal G$ in $\mathcal H$.

For a pattern complex $\mathcal G$, we call the densities of the cells
of the host complex $\mathcal H$ corresponding to edges of
$\mathcal G$ \emph{relevant}.  Thus, for each
$e\in E(\mathcal G^{(j)})$, the corresponding $j$-cell of
$\mathcal H$ has a relative density with respect to its underlying
$(j-1)$-complex.  These densities may vary from cell to cell, but only
a common positive lower bound will be required; cells not used by
$\mathcal G$ play no role.  The expected number of copies, or of
extensions of a fixed copy, is the standard product of the available
choices of vertices in each class and the relevant cell densities.
With this notation, we use the following two standard consequences of
regularity.

\begin{lemma}[Dense counting lemma,
{\cite[Lemma~6]{CooleyFountoulakisKuhnOsthus2009}}]
\label{lem:dense-counting}
Let $k,h,b$ be fixed positive integers, and let $\varepsilon,d>0$.
There exist $n_0$ and $\delta>0$ such that the following holds for every
$n\ge n_0$.  Let $\mathcal G$ be a $(k-1,h)$-complex on at most $b$
vertices, and let $\mathcal H$ be a $(k-1,h)$-complex with vertex
classes of size $n$ which respects the partition of $\mathcal G$.
Suppose that every relevant $j$-cell $C$ of $\mathcal H$, for
$2\le j\le k-1$, is $(d_C,\delta,1)$-regular with respect to its
underlying $(j-1)$-complex, where $d_C\ge d$.  Then
$|\mathcal G|_{\mathcal H}$ is equal to its expected value up to a
factor $1\pm\varepsilon$.
\end{lemma}

\begin{lemma}[Extension lemma,
{\cite[Lemma~5]{CooleyFountoulakisKuhnOsthus2009}}]
\label{lem:extension}
Let $k,h,b',b''$ be fixed positive integers with $b''<b'$, and let
$\beta,d>0$.  There exist $r,n_0$ and regularity parameters
$\delta,\delta_k>0$ such that the following holds for every
$n\ge n_0$.  Let $\mathcal G$ be a $(k,h)$-complex on at most $b'$
vertices, let $\mathcal G'$ be an induced subcomplex of $\mathcal G$ on
at most $b''$ vertices, and let $\mathcal H$ be a $(k,h)$-complex with
vertex classes of size $n$ which respects the partition of
$\mathcal G$.  Suppose that every relevant $j$-cell $C$ of
$\mathcal H$, for $2\le j\le k-1$, is
$(d_C,\delta,1)$-regular with respect to its underlying
$(j-1)$-complex, and every relevant $k$-cell $C$ is
$(d_C,\delta_k,r)$-regular with respect to its underlying
$(k-1)$-complex, where $d_C\ge d$.

Then all but at most
$\beta|\mathcal G'|_{\mathcal H}$ partition-respecting labelled copies
of $\mathcal G'$ extend to partition-respecting labelled copies of
$\mathcal G$.  Moreover, for every non-exceptional copy, the number of
extensions is within a factor $1\pm\beta$ of the expected number.
\end{lemma}

\subsection{Ordered pair-cells and regular slices}
\label{subsec:regular-slices}

We next introduce the ordered pair-cells used throughout the paper.
Let $V(H)=V_1\cup\cdots\cup V_m$ be a cluster partition, and put $\lambda_i:=\frac{|V_i|}{n}$ for $i\in[m]$. For $(i,j)\in[m]^2$, write
\[
    \overrightarrow V_{ij}
    =
    \begin{cases}
        V_i\times V_j, & i\ne j,\\
        \{(x,y)\in V_i^2:x\ne y\}, & i=j.
    \end{cases}
\]
An ordered pair partition of $\overrightarrow V_{ij}$ is written as
\[
    \overrightarrow V_{ij}
    =
    P_{ij}^1\cup\cdots\cup P_{ij}^{q_{ij}}.
\]
For a pair-cell $P_{ij}^a$, define its relative weight by
\[
    w(P_{ij}^a)
    :=
    \frac{|P_{ij}^a|}{|\overrightarrow V_{ij}|}.
\]
We always take the partitions coherently under reversal: if
$P_{ij}^a$ is a pair-cell over $(i,j)$, then
\[
 (P_{ij}^a)^\top
 :=
 \{(y,x):(x,y)\in P_{ij}^a\}
\]
is a pair-cell over $(j,i)$.  This may be enforced by a bounded common
refinement.

For pair-cells $P_{ij}^a$, $P_{jk}^b$ and $P_{ki}^c$, let
\[
    \mathcal T=(P_{ij}^a,P_{jk}^b,P_{ki}^c)
\]
denote the corresponding triad.  Its ordered support is
\[
K_3(\mathcal T)
:=
\bigl\{
    \{x,y,z\}:
    x,y,z\text{ are distinct},
    (x,y)\in P_{ij}^a,\ 
    (y,z)\in P_{jk}^b,\ 
    (z,x)\in P_{ki}^c
\bigr\}.
\]
The reduced weight of $\mathcal T$ is
\[
    w(\mathcal T)
    :=
    w(P_{ij}^a)w(P_{jk}^b)w(P_{ki}^c),
\]
and its $H$-density is
\[
    d_H(\mathcal T)
=
\frac{|E(H)\cap K_3(\mathcal T)|}{|K_3(\mathcal T)|}.
\]
When applying the regularity terminology, we view each ordered pair-cell
as the corresponding bipartite graph between the two cluster
occurrences; in this way, $\mathcal T$ determines a $3$-partite
$2$-complex. Fix $\delta_3>0$ and $r\in\mathbb N$.  A triad $\mathcal T$ is called
\emph{$H$-regular} if the restriction of $H$ to $K_3(\mathcal T)$ is
$(d_H(\mathcal T),\delta_3,r)$-regular with respect to the underlying
$2$-complex of $\mathcal T$.  For a fixed density threshold $d_0>0$, we
call $\mathcal T$ \emph{dense} if it is $H$-regular and
$d_H(\mathcal T)\ge d_0$.

The notation permits repeated cluster indices.  Whenever a bounded
pattern uses the same cluster more than once, we split that cluster into
a bounded number of almost equal subclusters and assign distinct
occurrences to distinct subclusters.  By the slicing lemma, all required
regularity and positive-density conditions are preserved, with the
losses absorbed into the parameter hierarchy. If a subset has relative density bounded below by a positive constant in
a pair-cell, averaging permits the corresponding refined pair-cell to
be chosen so that its relative density remains bounded below by a
positive constant. Thus repeated indices are only a notational convenience,
and every application of a counting or extension lemma involves
distinct cluster occurrences.

We shall use the following locally cleaned form of the regular slice
lemma.  It is a standard consequence of the error-function version of
the regular slice lemma of Allen, B{\"o}ttcher, Cooley and
Mycroft~\cite{ABCM17}, followed by a bounded cleaning and refinement.

\begin{lemma}[Regular slice with local cleaning]
\label{lem:regular-slice}
For every $\eta,\delta_2,\delta_3>0$ and integers $r,m_0$,
there exist integers $M_0$ and $n_0$ such that every $3$-graph $H$ on
$n\ge n_0$ vertices
has an equitable cluster partition
\[
    V(H)=V_1\cup\cdots\cup V_m,
    \qquad
    m_0\le m\le M_0,
\]
and ordered pair partitions
\[
    \overrightarrow V_{ij}
    =
    P_{ij}^1\cup\cdots\cup P_{ij}^{q_{ij}},
    \qquad
    (i,j)\in[m]^2,
\]
with the following properties.
\begin{enumerate}
    \item Each $q_{ij}$ is at most $M_0$, every pair-cell has relative
    weight at least $1/M_0$, and every pair-cell is
    $(w(P_{ij}^a),\delta_2,1)$-regular with respect to its complete
    bipartite support.

    \item For every ordered cluster triple $(i,j,k)$, the total reduced
    weight of triads which are not $H$-regular is at most $\eta$.

    \item For every pair-cell $P_{ij}^a$,
    \[
        \sum_{\substack{k\in[m],\,b,c\\
        (P_{ij}^a,P_{jk}^b,P_{ki}^c)
        \text{ is not }H\text{-regular}}}
        \lambda_k w(P_{jk}^b)w(P_{ki}^c)
        \le\eta.
    \]
\end{enumerate}
\end{lemma}

The lemma follows by applying the error-function version with an error
function $\varepsilon(t)$ satisfying
$\varepsilon(t)\ll\eta t^{-C}$ for every $t$, where $C$ is a
sufficiently large absolute constant.  Since there are only boundedly
many cluster triples and pair-cells, and every pair-cell has weight
bounded below in terms of $M_0$, the global exceptional estimates imply
the two local estimates.  The bounded refinements required for reversal
coherence and repeated cluster occurrences are absorbed by increasing
$M_0$.

\subsection{Reduced consequences}
\label{subsec:reduced-consequences}

We now record the two numerical consequences of the slice that will be
used on both sides of $1/3$.

\begin{lemma}[Slice counting facts]
\label{lem:slice-counting-facts}
Suppose that $0<\tau<p<1$ and that $d_0$, $\eta$, $\mu$ and the
regularity parameters are sufficiently small in terms of $p-\tau$.  Then the regular slice
may be chosen so that the following hold.
\begin{enumerate}
    \item For every ordered cluster triple $(i,j,k)$, the total reduced
    weight of dense triads supported on $(V_i,V_j,V_k)$ is at least
    $\tau$.

    \item If $\mathcal A_{ij}$, $\mathcal A_{jk}$ and
    $\mathcal A_{ki}$ are collections of pair-cells with total weights
    $a_{ij}$, $a_{jk}$ and $a_{ki}$, respectively, then
    \[
        \sum_{\substack{P_{ij}^a\in\mathcal A_{ij},\,
                        P_{jk}^b\in\mathcal A_{jk},\,
                        P_{ki}^c\in\mathcal A_{ki}}}
        w(P_{ij}^a)w(P_{jk}^b)w(P_{ki}^c)
        =
        a_{ij}a_{jk}a_{ki}.
    \]
    In particular, the total reduced weight of any subfamily of these
    triads is at most $a_{ij}a_{jk}a_{ki}$.
\end{enumerate}
\end{lemma}

\begin{proof}
By the repeated-cluster convention, it suffices to consider distinct
cluster occurrences.  Applying $(p,\mu)$-density to the three cluster
sets gives
\[
    \frac{e_H(V_i,V_j,V_k)}
         {|V_i||V_j||V_k|}
    \ge p-o(1),
\]
since the cluster sizes are bounded below by a positive multiple of
$n$ and $\mu$ is sufficiently small.

By the dense counting lemma, the normalised size of the support of each
triad is equal to its reduced weight up to $o(1)$, uniformly over the
bounded family of triads.  Triads which are not $H$-regular have total
reduced weight at most $\eta$, while $H$-regular triads of density less
than $d_0$ contribute at most $d_0+o(1)$ to the normalised edge count.
Hence the total reduced weight of dense triads is at least
\[
    p-d_0-\eta-o(1)\ge\tau.
\]
This proves the first assertion.  The second follows by expanding the
product of the three sums of pair-cell weights.
\end{proof}

For a pair-cell $P_{ij}^a$, define its \emph{dense reduced extension
weight} by
\[
    \operatorname{ext}(P_{ij}^a)
    :=
    \sum_{\substack{k\in[m],\,b,c\\
    (P_{ij}^a,P_{jk}^b,P_{ki}^c)\text{ is dense}}}
    \lambda_k w(P_{jk}^b)w(P_{ki}^c).
\]

\begin{lemma}[Reduced codegree inheritance]
\label{lem:reduced-codegree}
Let $\Delta<\alpha$, and suppose that
$\delta_2(H)\ge\alpha n$.  If $d_0$, $\eta$ and the regularity
parameters are sufficiently small in terms of $\alpha-\Delta$, then
every pair-cell $P_{ij}^a$ satisfies
\[
    \operatorname{ext}(P_{ij}^a)\ge\Delta.
\]
\end{lemma}

\begin{proof}
Suppose, to the contrary, that some pair-cell
$P=P_{ij}^a$ satisfies $\operatorname{ext}(P)<\Delta$.  We estimate $\sum_{(x,y)\in P}d_H(x,y)$ by separating the contributions from dense triads, triads which are not
$H$-regular, and $H$-regular triads of density less than $d_0$
containing $P$.

By the dense counting lemma, the total number of triples in the supports
of dense triads containing $P$ is at most
\[
    \bigl(\operatorname{ext}(P)+o(1)\bigr)n|P|.
\]
The local cleaning property gives total support at most
\[
    (\eta+o(1))n|P|
\]
for triads containing $P$ which are not $H$-regular.  Finally, the
$H$-regular triads of density less than $d_0$ contribute at most
\[
    (d_0+o(1))n|P|
\]
actual $H$-edges.  Consequently,
\[
\begin{aligned}
    \sum_{(x,y)\in P}d_H(x,y)
    &\le
    \bigl(\operatorname{ext}(P)+\eta+d_0+o(1)\bigr)n|P|\\
    &<
    \alpha n|P|,
\end{aligned}
\]
provided the parameters are sufficiently small.  This contradicts
$\delta_2(H)\ge\alpha n$.
\end{proof}

The preceding lemmas contain the regularity input common to both
ranges of the proof.  Above $1/3$, we shall use the ordered pair-cells
directly as the states of a directed reduced graph.  At and below
$1/3$, we first need to retain the additional information carried by
the tight components of the original $3$-graph; the corresponding
refinement is described in the next subsection.

\subsection{Component-coloured refinements}
\label{subsec:component-coloured-refinements}

For the range at and below $1/3$, we need to retain one additional
piece of information when passing to the regular slice, namely the
actual tight component whose shadow contains each pair.  We explain
here how this information can be incorporated into the regularity
framework developed above.

For a tight component $C$ of a $3$-graph $H$, let
$\partial C:=\{uv\in\binom{V(H)}2:uv\subseteq e\text{ for some }
e\in E(C)\}$ denote its shadow.  The shadows of distinct tight
components are pairwise disjoint, since two edges containing the same
pair are tightly adjacent.  In particular, whenever
$\delta_2(H)>0$, every pair belongs to the shadow of a unique tight
component.

\begin{lemma}[Component-coloured locally cleaned slice]
\label{lem:component-coloured-slice}
Fix $p,\alpha,\eta,\chi>0$, a bounded set of distinguished vertices,
and a bounded initial vertex partition.  Let $\zeta>0$ be sufficiently
small in terms of $\eta$ and $\chi$.  Then, for all sufficiently small
$\mu>0$ and all sufficiently large $n$, every $n$-vertex
$(p,\mu)$-dense $3$-graph $H$ satisfying
$\delta_2(H)\ge\alpha n$ admits a partition
\[
    V(H)=V_0\dot\cup V_1\dot\cup\cdots\dot\cup V_t,
    \qquad
    |V_0|\le\zeta n,
\]
with all distinguished vertices contained in $V_0$, and a cleaned
ordered pair-state system on $I=[t]$, with retained state families
$\mathcal A_{ij}$ and retained triad families $\mathcal T_{ijk}$,
satisfying the following properties.
\begin{enumerate}[label=\textup{(R\arabic*)},leftmargin=2.8em]

\item Every retained state is a regular pair-cell contained in the
shadow of one actual tight component, and every retained triad uses
three states carrying the same actual component colour.

\item Writing $\lambda_i:=|V_i|/(n-|V_0|)$, the retained state weights
are normalised so that $\sum_{a\in\cA_{ij}}w(a)=1$ for every ordered
cluster pair $(i,j)$.  Moreover, for every ordered
cluster triple $(i,j,k)$,
\[
    \sum_{(a,b,c)\in\mathcal T_{ijk}}
        w(a)w(b)w(c)
    \ge p-\eta,
\]
and every retained state $a\in\mathcal A_{ij}$ satisfies
\[
    \sum_{k\in I}\lambda_k
    \sum_{\substack{b\in\mathcal A_{jk},\,
                    c\in\mathcal A_{ki}\\
                    (a,b,c)\in\mathcal T_{ijk}}}
        w(b)w(c)
    \ge\alpha-\eta.
\]

\item Every retained triad is vertex-lower-regular: any three subsets
of its tagged classes, each of relative size at least $\chi$, contain
an edge of the actual tight component
carried by that triad, provided the regularity error is sufficiently
small.

\item The construction respects every prescribed bounded refinement
of the vertex and pair partitions.  It can also be performed jointly
for two one-root deletions.  More precisely, for two distinguished
vertices $x,y$, the common pair refinement may be taken according to
\[
    \bigl(
        \operatorname{col}_{H-x}(uv),
        \operatorname{col}_{H-y}(uv)
    \bigr).
\]
Consequently, a common positive state carrying colours $C_x$ and
$C_y$ in the two projected systems is a positive-density family of
actual pairs contained in $\partial C_x\cap\partial C_y$.

\item Repeated cluster indices are realised by three coherent tagged
copies; reversal preserves weights and
\[
    (a,b,c)\in\mathcal T_{ijk}
    \quad\Longleftrightarrow\quad
    (a^\top,c^\top,b^\top)\in\mathcal T_{jik}.
\]
All six permutations of the tagged copies are obtained from one common
$S_3$-invariant refinement.

\end{enumerate}
\end{lemma}

\begin{proof}
We first bound the number of actual component colours.  Let $C$ be a
tight component and put $Q:=\partial C$.  Choose $uv\in Q$.  Since
$uv$ lies in the shadow of $C$, every edge of $H$ containing $uv$
belongs to $C$.  Put $Z:=N_H(u,v)$.  Then $|Z|\ge\alpha n$, and for
every $z\in Z$ we have $uz,vz\in Q$.

Fix $z\in Z$.  Since $uz\in Q$, every edge containing $uz$ also
belongs to $C$.  Hence, for every $w\in N_H(u,z)$, we have $zw\in Q$.
It follows that $d_Q(z)\ge d_H(u,z)\ge\alpha n$.  Summing over
$z\in Z$, we obtain
\[
    2e(Q)
    \ge
    \sum_{z\in Z}d_Q(z)
    \ge
    \alpha^2n^2.
\]
Thus every nonempty component shadow contains at least
$\alpha^2n^2/2$ pairs.  Since distinct component shadows are
edge-disjoint, the number of actual component colours is bounded in
terms of $\alpha$.

Colour every pair by the unique actual tight component whose shadow
contains it.  We first apply simultaneous multicolour graph regularity
to these boundedly many shadow graphs, refining the prescribed vertex
and pair partitions.  We then apply the error-function regular-slice
lemma relative to the resulting bounded level-two refinement, followed
by the local cleaning from \Cref{lem:regular-slice}, with local error
$\eta_{\rm reg}$.  By the choice of $\zeta$, we may choose
$\eta_{\rm reg}$ and $d_0$ sufficiently small so that all losses are
absorbed into the target error $\eta$, with the regularity parameters
chosen sufficiently small in terms of $\chi$. Thus every
retained pair-cell is contained in one actual component colour.

An actual edge puts all three of its pairs in the shadow of the same
tight component.  Consequently, a triad whose three pair-cells do not
carry the same actual colour contains no edge of $H$.  We therefore
retain only regular monochromatic triads of density at least $d_0$.

For a cluster pair $ij$, let $m(a)$ denote the raw relative mass of a
retained state $a$, and put
$\sigma_{ij}:=\sum_{a\in\mathcal A_{ij}}m(a)$.  The local cleaning may
be chosen so that $1-\zeta\le\sigma_{ij}\le1$ for every ordered
cluster pair.  Define $w(a):=m(a)/\sigma_{ij}$.  Thus the retained state weights over every ordered cluster pair sum
exactly to one.  Since $\sigma_{ij}=1-O(\zeta)$, this normalisation changes every
product of a bounded number of state weights by a factor
$1+O(\zeta)$.  The losses coming from deleted pair-cells, irregular
triads, low-density regular triads and the subsequent normalisation
are therefore bounded by
$O(\zeta+\eta_{\rm reg}+d_0)$. 
By the choices above, the slice-counting argument gives
\[
    \sum_{(a,b,c)\in\cT_{ijk}}
        w(a)w(b)w(c)\ge p-\eta,
\]
for every ordered cluster triple $(i,j,k)$.

Similarly, fix a retained state $a\in\mathcal A_{ij}$.  Summing
$d_H(u,v)$ over the actual pairs $(u,v)$ represented by $a$, the local
codegree calculation from \Cref{lem:reduced-codegree} shows that the contribution lost through deleted pair-cells, irregular triads
and regular triads of density below $d_0$ is
$O(\zeta+\eta_{\rm reg}+d_0)$.  
Since every edge containing a pair represented by $a$ has the same
actual component colour as $a$, all remaining positive contributions
come from retained monochromatic triads.  By the same choices, and after accounting for the normalisation, we
obtain
\[
    \sum_{k\in I}\lambda_k
    \sum_{\substack{b\in\cA_{jk},\,c\in\cA_{ki}\\
                    (a,b,c)\in\cT_{ijk}}}
        w(b)w(c)\ge\alpha-\eta.
\]
This proves \textup{(R1)} and \textup{(R2)}.

Every retained pair-cell has density bounded below by $\omega_0$, and
every retained triad has $H$-density bounded below by $d_0$.  The
slicing lemma and the dense counting lemma therefore imply
\textup{(R3)}.

For two one-root deletions, we begin with the joint colour vector
\[
    \bigl(
        \operatorname{col}_{H-x}(uv),
        \operatorname{col}_{H-y}(uv)
    \bigr)
\]
on $V(H)\setminus\{x,y\}$ and regularise the two induced $3$-graphs
on one common refinement.  Both coordinate colourings use only
boundedly many colours, so this is again a bounded prescribed
level-two refinement.  Projecting to either coordinate gives the
component-coloured slice for the corresponding deletion.  In
particular, a common state of joint colour $(C_x,C_y)$ consists of
actual pairs lying simultaneously in $\partial C_x$ and
$\partial C_y$.

Finally, whenever repeated cluster indices occur, replace each cluster
by three tagged copies, regularise all six permutations of the tags
simultaneously, and take one common $S_3$-invariant refinement.
Reversal is imposed as part of the same common refinement.  This gives
\textup{(R4)} and \textup{(R5)}.
\end{proof}

After fixing a positive margin above the threshold, we use the
hierarchy
\begin{equation}
\label{eq:cleaning-hierarchy}
    \frac1n\ll\mu\ll\varepsilon
    \ll\zeta,\eta_{\rm reg}
    \ll d_0
    \ll\eta
    \ll\omega_0,\lambda_0
    \ll\chi,p,\alpha.
\end{equation}
Here $\eps$ is the graph-regularity error, while $\omega_0$ and
$\lambda_0$ are fixed lower bounds on the weights of retained states
and clusters, respectively.
Only clusters, states and triads above the displayed cleaning
thresholds are retained.

The dense counting lemma also allows us to realise every fixed finite
pattern carried by the cleaned state system.  We shall use the
following consequence repeatedly.

\begin{lemma}[Embedding a finite state pattern]
\label{lem:finite-state-pattern}
Let $\mathcal W$ be a fixed finite tight-walk pattern whose consecutive
pair positions are assigned retained states and whose edges are
assigned retained triads, consistently with the transition digraph.
Then $\mathcal W$ has a homomorphic embedding in the actual tight
component carrying these states.  The assertion remains valid for
closed patterns and for repeated cluster indices.
\end{lemma}

\begin{proof}
Use coherent tagged copies to separate repeated cluster occurrences.
The assigned states and triads form a fixed regular $3$-complex in
which every relevant pair density is bounded below by $\omega_0$ and
every relevant triad density is bounded below by $d_0$.  The dense
counting lemma therefore gives a positive number of
partition-respecting labelled realisations of the pattern.  Since all
assigned states and triads have the same actual component colour,
every hyperedge in such a realisation belongs to that actual tight
component.  Pulling the tagged classes back to their original clusters
gives the required homomorphic embedding.
\end{proof}

\section{Connecting lemma}
\label{sec:connecting}

In this section we prove Lemma~\ref{lem:connecting}.  We form a directed
graph whose vertices are the ordered pair-cells obtained in
Section~\ref{sec:regularity}.  A dense triad gives a transition between
two consecutive pair-states of a tight path.  We prove that this
pair-state digraph is strongly connected and aperiodic, obtain reduced
walks of one fixed length, and lift these walks to tight paths in the
original 3-graph.

Throughout this section, fix $p>1/3$ and $\alpha>\thr_4(p)$. Choose constants $\frac13<\tau<\tau'<p$
so that, if $a=a(\tau)$ is the smaller root of $a^3+(1-a)^3=\tau$, then $a^2<\alpha$.  Choose $\Delta$ with $a^2<\Delta<\alpha$. We choose $d_0$, $\eta$, $\mu$ and all regularity parameters sufficiently small in terms of the gaps in the inequalities above.

\subsection{The reduced pair-state digraph and lifting}
\label{subsec:pair-state-lifting}

Let $H$ be an $n$-vertex $(p,\mu)$-dense 3-graph with
$\delta_2(H)\ge\alpha n$, and apply Lemma~\ref{lem:regular-slice}.
We use the ordered pair-cells of the
resulting slice to encode the evolution of a tight path.

The vertices of the \emph{reduced pair-state digraph} $D$ are the
pair-cells $P_{ij}^a$.  We place a directed edge
\[
    P_{ij}^a\longrightarrow P_{jk}^b
\]
whenever there exists a pair-cell $P_{ki}^c$ such that $(P_{ij}^a,P_{jk}^b,P_{ki}^c)$ is a dense triad.  The direction records the change in the terminal
ordered pair of a tight path: extending a path ending in $(x,y)$ by a
vertex $z$ replaces its terminal pair by $(y,z)$.  Moreover, cyclic
relabelling preserves regularity and density, so every dense triad gives
the three transitions
\[
    P_{ij}^a\to P_{jk}^b,\qquad
    P_{jk}^b\to P_{ki}^c,\qquad
    P_{ki}^c\to P_{ij}^a.
\]

The main reduced statement is the following.

\begin{lemma}[Reduced pair-state theorem]
\label{lem:reduced-pair-state-main}
If $\mu$ and the regularity parameters are sufficiently small in terms
of $p$ and $\alpha$, then the reduced pair-state digraph $D$ is strongly
connected and aperiodic.
\end{lemma}

We postpone the proof of
Lemma~\ref{lem:reduced-pair-state-main} to
Subsection~\ref{subsec:reduced-connectivity}.  We first show how it
implies Lemma~\ref{lem:connecting}.  We shall use the following standard
fact about primitive directed graphs.

\begin{lemma}[Fixed-length walks, {\cite{Neufeld1996}}]
\label{lem:fixed-length-walks}
For every integer $M$ there exists an integer $L$ such that the
following holds.  If $D$ is a strongly connected aperiodic directed
graph with at most $M$ vertices, then, for any two vertices $P,Q$ of
$D$ and every integer $\ell\ge L$, there is a directed walk of length
exactly $\ell$ from $P$ to $Q$.  We may assume that $L\ge4$.
\end{lemma}

We next associate states of $D$ with the prescribed endpoint pairs.
For two disjoint ordered pairs $(a,b)$ and $(c,d)$, define
\[
\begin{aligned}
    S_{ab}
    &:=
    \{(x,y):a,b,x,y\text{ are distinct and }
             abx,bxy\in E(H)\},\\
    T_{cd}
    &:=
    \{(u,v):u,v,c,d\text{ are distinct and }
             uvc,vcd\in E(H)\}.
\end{aligned}
\]
By extending the corresponding endpoint pair twice and using
$\delta_2(H)\ge\alpha n$, we obtain
\[
    |S_{ab}|,\ |T_{cd}|
    \ge
    (\alpha^2-o(1))n^2.
\]
Since the number of pair-cells is bounded, there exists a constant
$\sigma>0$, depending only on $\alpha$ and the complexity of the slice,
such that $S_{ab}$ has relative density at least $\sigma$ in some
pair-cell.  We call any such pair-cell a \emph{starting state} for
$(a,b)$.  Similarly, a pair-cell in which $T_{cd}$ has relative density
at least $\sigma$ is called a \emph{terminal state} for $(c,d)$.

The following lemma lifts a bounded reduced walk between these endpoint
states to a tight path in $H$.

\begin{lemma}[Lifting reduced walks]
\label{lem:lifting-reduced-walks}
Let $P$ be a starting state for $(a,b)$ and let $Q$ be a terminal state
for $(c,d)$.  Suppose that there is a directed walk of length $r\ge4$
from $P$ to $Q$ in $D$, where $r$ is bounded.  Then $H$ contains a
tight path from $(a,b)$ to $(c,d)$ with exactly $r+2$ internal
vertices.
\end{lemma}

\begin{proof}
Let $P_0,P_1,\ldots,P_r$ be the reduced walk, where $P_0=P$ and $P_r=Q$.  We may write
\[
    P_q=P_{i_qi_{q+1}}^{a_q},
    \qquad q=0,\ldots,r.
\]
For each $q=0,\ldots,r-1$, choose a pair-cell $C_q=P_{i_{q+2}i_q}^{c_q}$ such that $(P_q,P_{q+1},C_q)$ is a dense triad.

These cells define a labelled tight-path complex $\mathcal G$ on
vertices $z_0,z_1,\ldots,z_{r+1}$.
For $q=0,\ldots,r$, the ordered pair $(z_q,z_{q+1})$ is assigned to
$P_q$.  For $q=0,\ldots,r-1$, the ordered pair
$(z_{q+2},z_q)$ is assigned to $C_q$, and
$z_qz_{q+1}z_{q+2}$ is assigned to the top-level cell of $H$ in the
dense triad $(P_q,P_{q+1},C_q)$.

By the repeated-cluster convention from
Subsection~\ref{subsec:regular-slices}, repeated cluster occurrences may
be separated while preserving the required regularity and positive
densities.  The refined endpoint cells may also be chosen so that
$S_{ab}\cap P_0$ and $T_{cd}\cap P_r$ retain relative density bounded
below by a positive constant.

Let $\mathcal H_{\rm walk}$ be the regular host complex formed by the
selected refined pair-cells and dense triads, and let $\mathcal G'$ be
the subcomplex of $\mathcal G$ induced by $z_0,z_1,z_r,z_{r+1}$. Since $r\ge4$, the two endpoint pairs are separated in $\mathcal G'$.
Call a partition-respecting labelled copy of $\mathcal G'$ admissible if
\[
    (z_0,z_1)\in S_{ab}\cap P_0
    \quad\text{and}\quad
    (z_r,z_{r+1})\in T_{cd}\cap P_r.
\]
The endpoint density bounds imply that the admissible copies form a
positive proportion of all partition-respecting copies of
$\mathcal G'$ in $\mathcal H_{\rm walk}$.

Choose the error parameter in Lemma~\ref{lem:extension} smaller than
this proportion.  The extension lemma then guarantees that at least one
admissible copy of $\mathcal G'$ extends to a copy of $\mathcal G$.
Deleting the four fixed vertices $a,b,c,d$ from the host classes before
the application changes the relevant densities only negligibly, so the
extension may be chosen disjoint from them.

Consequently, there exist distinct vertices $x_0,x_1,\ldots,x_{r+1}$ realising $\mathcal G$, with
\[
    (x_0,x_1)\in S_{ab}
    \quad\text{and}\quad
    (x_r,x_{r+1})\in T_{cd}.
\]
The endpoint conditions give the four boundary edges
\[
    abx_0,\qquad bx_0x_1,\qquad
    x_rx_{r+1}c,\qquad x_{r+1}cd,
\]
while the copy of $\mathcal G$ gives $x_qx_{q+1}x_{q+2}\in E(H)$ for every $q=0,\ldots,r-1$.
Hence
\[
    a,b,x_0,x_1,\ldots,x_{r+1},c,d
\]
is a tight path from $(a,b)$ to $(c,d)$ with exactly $r+2$ internal
vertices.
\end{proof}

We are now ready to deduce the connecting lemma from Lemma~\ref{lem:reduced-pair-state-main}.  

\begin{proof}[Proof of Lemma~\ref{lem:connecting}, assuming
Lemma~\ref{lem:reduced-pair-state-main}]
Fix $p>1/3$ and $\alpha>\thr_4(p)$.  Let $M$ be a uniform upper bound on
the number of states in any reduced pair-state digraph obtained from
Lemma~\ref{lem:regular-slice}.  Let $L=L(M)$ be given by
Lemma~\ref{lem:fixed-length-walks}, and set $t:=L+2$.
Choose $\mu_0>0$ sufficiently small and $n_0$ sufficiently large so that
all applications of the regular slice and
Lemma~\ref{lem:lifting-reduced-walks} below are valid.

Let $H$ be an $n$-vertex $(p,\mu)$-dense 3-graph with
$0<\mu\le\mu_0$, $n\ge n_0$, and $\delta_2(H)\ge\alpha n$.  Let
$(a,b)$ and $(c,d)$ be two disjoint ordered pairs.  Apply
Lemma~\ref{lem:regular-slice} and form the corresponding reduced
pair-state digraph $D$.  Choose a starting state $P$ for $(a,b)$ and a
terminal state $Q$ for $(c,d)$.

By Lemma~\ref{lem:reduced-pair-state-main}, the digraph $D$ is strongly
connected and aperiodic.  Hence Lemma~\ref{lem:fixed-length-walks}
provides a directed walk of length exactly $L$ from $P$ to $Q$.
Lemma~\ref{lem:lifting-reduced-walks} then yields a tight path from
$(a,b)$ to $(c,d)$ with exactly $L+2=t$
internal vertices.  This proves the lemma.
\end{proof}

It remains to prove Lemma~\ref{lem:reduced-pair-state-main}.  In the
next subsection, we prove that the reduced pair-state digraph is strongly
connected and aperiodic.

\subsection{Strong connectivity and aperiodicity}
\label{subsec:reduced-connectivity}

The proof of Lemma~\ref{lem:reduced-pair-state-main} uses the following
finite scalar statement.

\begin{lemma}[Reduced scalar lemma]
\label{lem:scalar-main}
Let $I$ be a finite non-empty set, let $1/3<\tau<1$, let $a=a(\tau)\in(0,1/3)$ be the smaller root of
$a^3+(1-a)^3=\tau$, and let $\Delta>a^2$. Put $\Phi(x,y,z)=xyz+(1-x)(1-y)(1-z)$. There do not exist positive weights
$(\lambda_i)_{i\in I}$ summing to one and a non-trivial matrix
$(s_{ij})_{i,j\in I}\in[0,1]^{I\times I}$ satisfying the following
conditions:
\begin{enumerate}
    \item[(1)] For all $i,j,k\in I$, $\Phi(s_{ij},s_{jk},s_{ki})\ge \tau$.
    \item[(2)] For all $i,j\in I$ with $s_{ij}>0$, $\sum_{k\in I}\lambda_k s_{jk}s_{ki}\ge \Delta$.
    \item[(3)] For all $i,j\in I$ with $s_{ij}<1$, $\sum_{k\in I}\lambda_k(1-s_{jk})(1-s_{ki})\ge \Delta$.
\end{enumerate}
Here non-trivial means that the matrix is neither identically zero nor
identically one.
\end{lemma}

We postpone the proof of Lemma~\ref{lem:scalar-main} to
Subsection~\ref{subsec:scalar-proof}.

\begin{proof}[Proof of Lemma~\ref{lem:reduced-pair-state-main}]
We first prove strong connectivity.  Suppose that $D$ is not strongly
connected, and let $\mathcal C$ be a non-empty proper set of states with
no directed edge from $\mathcal C$ to its complement.  For every
ordered cluster pair $(i,j)$, define
\[
    s_{ij}
    :=
    \sum_{P_{ij}^a\in\mathcal C}w(P_{ij}^a).
\]

For a dense triad, the three cyclic transitions among its states belong
to $D$.  Since no transition leaves $\mathcal C$, the three states of
every dense triad are either all in $\mathcal C$ or all outside
$\mathcal C$.  By Lemma~\ref{lem:slice-counting-facts}, the total
reduced weight of dense triads on any ordered cluster triple is at least
$\tau'$.  The total weight of all triads whose states lie in
$\mathcal C$ is $s_{ij}s_{jk}s_{ki}$,
and the corresponding weight outside $\mathcal C$ is $(1-s_{ij})(1-s_{jk})(1-s_{ki})$.
Consequently, 
\[
    \Phi(s_{ij},s_{jk},s_{ki})
    \ge
    \tau'
    >
    \tau
\]
for every $i,j,k\in[m]$.

Suppose that $s_{ij}>0$.  Then some state
$P_{ij}^a\in\mathcal C$ has positive weight.  By
Lemma~\ref{lem:reduced-codegree}, $\operatorname{ext}(P_{ij}^a)\ge\Delta$.
Every dense triad containing $P_{ij}^a$ has its other two states in
$\mathcal C$, and hence
\[
    \sum_{k\in[m]}\lambda_k s_{jk}s_{ki}
    \ge
    \Delta.
\]
Similarly, if $s_{ij}<1$, choose a positive-weight state outside
$\mathcal C$.  Any dense triad containing this state has its other two
states outside $\mathcal C$, and therefore
\[
    \sum_{k\in[m]}
    \lambda_k(1-s_{jk})(1-s_{ki})
    \ge
    \Delta.
\]
Thus the matrix $(s_{ij})$ satisfies
\textit{(1)}--\textit{(3)} of Lemma~\ref{lem:scalar-main}.  Since $\mathcal C$ is non-empty and
proper and every pair-cell has positive weight, this matrix is
non-trivial, contradicting Lemma~\ref{lem:scalar-main}.  Hence $D$ is
strongly connected.

It remains to prove that $D$ is aperiodic.  Recall that the period of a
strongly connected directed graph is the greatest common divisor of the
lengths of its closed directed walks.  For every cluster $V_i$,
Lemma~\ref{lem:slice-counting-facts} gives a dense triad supported on
$(V_i,V_i,V_i)$.  Its three cyclic transitions form a closed directed
walk of length $3$, so the period of $D$ divides $3$.

Suppose that the period is $3$.  Then the states admit a cyclic
partition
\[
    \mathcal C_0\cup\mathcal C_1\cup\mathcal C_2
\]
such that every directed edge goes from $\mathcal C_r$ to
$\mathcal C_{r+1}$, with indices modulo $3$.  Fix $i\in[m]$, and let
$s_i^r$ be the total weight of pair-cells over $(i,i)$ belonging to
$\mathcal C_r$.  Then
\[
    s_i^0+s_i^1+s_i^2=1.
\]
If the first state of a dense triad supported on $(V_i,V_i,V_i)$ lies
in $\mathcal C_r$, then its second and third states lie in
$\mathcal C_{r+1}$ and $\mathcal C_{r+2}$, respectively.  Hence the
total reduced weight of such dense triads is at most
\[
    \sum_{r=0}^2 s_i^r s_i^{r+1}s_i^{r+2}
    =
    3s_i^0s_i^1s_i^2
    \le
    \frac19.
\]
This contradicts the lower bound $\tau'>1/3$ from
Lemma~\ref{lem:slice-counting-facts}.  Therefore the period is not
$3$.  Since it divides $3$, it is equal to $1$, and $D$ is aperiodic.
\end{proof}

\subsection{Proof of the reduced scalar lemma}
\label{subsec:scalar-proof}

The argument in this subsection is independent of the regularity
reduction.  We first use the triangle inequalities to constrain the
diagonal entries and exclude matrices with both low and high diagonal
entries.  In the remaining cases, an extremal off-diagonal entry,
together with the corresponding extension inequality, yields the
desired contradiction.  The symmetry under replacing $s_{ij}$ by
$1-s_{ij}$ reduces the all-high case to the all-low case.

\begin{proof}[Proof of Lemma~\ref{lem:scalar-main}]
For $i,j\in I$, write
\[
    D_{ij}^+
    :=
    \sum_{k\in I}\lambda_k s_{jk}s_{ki},
    \qquad
    D_{ij}^-
    :=
    \sum_{k\in I}\lambda_k(1-s_{jk})(1-s_{ki}).
\]

Putting $i=j=k$ in \textit{(1)} gives
\[
    s_{ii}^3+(1-s_{ii})^3\ge\tau.
\]
Since $a$ is the smaller root of
$x^3+(1-x)^3=\tau$, it follows that $s_{ii}\in[0,a]\cup[1-a,1]$ for every $i\in I$.

We first show that the two types cannot both occur.  Suppose that
$s_{ii}=d\le a$ and $s_{jj}=e\ge1-a$ for some $i,j\in I$.  Put
\[
    u=s_{ij},\qquad v=s_{ji},\qquad
    P=uv,\qquad Q=(1-u)(1-v).
\]
Applying \textit{(1)} to $(i,i,j)$ and $(i,j,j)$ gives
\[
    dP+(1-d)Q\ge\tau,
    \qquad
    eP+(1-e)Q\ge\tau.
\]
If $P\ge Q$, then $u+v\ge1$, and
\[
    dP+(1-d)Q
    \le
    aP+(1-a)Q
    =
    uv+(1-a)(1-u-v).
\]
Writing $s=u+v$, the last expression is at most
$s^2/4+(1-a)(1-s)$.  This is a convex function of
$s\in[1,2]$, so its maximum is $\max\{1/4,a\}<\tau$, a
contradiction.
If $Q\ge P$, then $u+v\le1$, and similarly
\[
    eP+(1-e)Q
    \le
    (1-a)P+aQ
    =
    uv+a(1-u-v).
\]
Writing $s=u+v$, the last expression is at most
$s^2/4+a(1-s)$.  This is convex on $[0,1]$, and hence is at most
$\max\{a,1/4\}<\tau$, again a contradiction.  Thus either
\[
    s_{ii}\le a\quad\text{for every }i\in I,
\]
or
\[
    s_{ii}\ge1-a\quad\text{for every }i\in I.
\]

Assume first that all diagonal entries are at most $a$.  We claim that
\[
    s_{ij}\le1-\tau
    \qquad\text{for all }i,j\in I.
\]
Indeed, fix $i,j$ and put
$d=s_{ii}\le a$, $x=s_{ij}$ and $y=s_{ji}$.  If $x>1-\tau$, then
$\Phi(d,x,y)$ is linear in $y$, while
\[
    \Phi(d,x,0)=(1-d)(1-x)<\tau
\]
and
\[
    \Phi(d,x,1)=dx\le d\le a<\tau.
\]
This contradicts \textit{(1)}.

Let $M:=\max_{i,j\in I}s_{ij}$.
If $M=0$, then the matrix is identically zero.  Suppose that $M>0$, and choose $i,j$ with $s_{ij}=M$.  By \textit{(2)},
\[
    D_{ij}^+\ge\Delta.
\]
We show that $s_{jk}s_{ki}\le a^2$ for every $k\in I$.

Fix $k\in I$, and put
\[
    b=s_{jk},\qquad c=s_{ki},\qquad w=\sqrt{bc}.
\]
Since $b,c\le M$, the assertion is immediate if $M\le a$.  Assume that
$a<M\le1-\tau$.  By \textit{(1)},
\[
    \tau
    \le
    \Phi(M,b,c)
    =
    (1-M)(1-b-c)+bc.
\]
Since $b+c\ge2w$, it follows that
\[
    g_M(w)
    :=
    w^2-2(1-M)w+(1-M)-\tau
    \ge0.
\]
Using $\tau=a^3+(1-a)^3$, we obtain
\[
    g_M(a)=(1-2a)(a-M)<0
\]
and
\[
    g_M(M)=M^3+(1-M)^3-\tau<0.
\]
Indeed, $a<M\le1-\tau<1-a$, and
$x^3+(1-x)^3<\tau$ for $x\in(a,1-a)$.  Since $g_M$ is convex, it is
negative throughout $[a,M]$.  As $0\le w\le M$ and $g_M(w)\ge0$, we
must have $w<a$, and hence $bc<a^2$.

Therefore
\[
    D_{ij}^+
    =
    \sum_{k\in I}\lambda_k s_{jk}s_{ki}
    \le
    a^2,
\]
contradicting $D_{ij}^+\ge\Delta>a^2$.  Hence the all-low case forces
$s_{ij}=0$ for every $i,j$.

Finally, suppose that all diagonal entries are at least $1-a$.  Define $t_{ij}:=1-s_{ij}$.
Then $t_{ii}\le a$ for every $i$, and
\[
    \Phi(t_{ij},t_{jk},t_{ki})
    =
    \Phi(s_{ij},s_{jk},s_{ki}).
\]
Moreover, the two extension conditions are interchanged under this
substitution.  Thus $(t_{ij})$ satisfies the same three assumptions.
By the all-low case, $t_{ij}=0$ for every $i,j$, and hence
$s_{ij}=1$ for every $i,j$.  This contradicts the assumed
non-triviality and completes the proof.
\end{proof}

\begin{remark}
\label{rem:scalar-sharpness}
The bound $\Delta>a(\tau)^2$ is sharp for the scalar lemma.  Indeed, the
constant matrix $s_{ij}\equiv a(\tau)$ satisfies
\[
    \Phi(s_{ij},s_{jk},s_{ki})=\tau,
    \qquad
    D_{ij}^+=a(\tau)^2,
    \qquad
    D_{ij}^-=(1-a(\tau))^2.
\]
Thus a non-trivial scalar obstruction exists at
$\Delta=a(\tau)^2$.  As $\tau$ approaches $1/3$ from above, the smaller root $a(\tau)$
approaches $1/3$ from below.  Consequently, the corresponding scalar
codegree threshold $a(\tau)^2$ tends to $1/9$.
\end{remark}

\section{The local Hamilton framework}
\label{sec:canonical-lifting}

In this section, we complete the proof in the range at and below
$1/3$, assuming the finite structural theorem stated below.  The
component-coloured locally cleaned slice from
\Cref{lem:component-coloured-slice} already contains all regularity
information that will be needed.  We first formulate the associated
state system and state the canonical component theorem.  Then we lift
the canonical reduced component to the original $3$-graph and verify
the four properties of a Hamilton framework.  The proof of the
canonical component theorem is deferred to
Section~\ref{sec:canonical-proof}.

\subsection{The state system and the canonical component theorem}
\label{subsec:canonical-state-system}

We use the state system supplied by
\Cref{lem:component-coloured-slice}.  We record only the finite
information from this system that will be needed below.  After fixing
the cleaning bounds, the resulting system has bounded complexity and
fixed positive lower bounds on all retained cluster, state and triad
weights, while preserving reversal symmetry, the actual
tight-component colour of every retained state, and the joint
refinements required for two-root consistency.  These cleaning bounds
remain fixed throughout Sections~6 and~7 as the admissibility error
$\eta$ tends to zero.

Let $I$ be the cluster index set, and let $\lambda_i>0$ be the
normalised weight of the cluster $V_i$, so that
$\sum_{i\in I}\lambda_i=1$.  For every ordered pair $(i,j)$, let
$\cA_{ij}$ be the set of retained pair states over $(V_i,V_j)$.
Each state $a\in\cA_{ij}$ has a positive weight $w(a)$, and
$\sum_{a\in\cA_{ij}}w(a)=1$.  Reversal identifies $\cA_{ij}$ with
$\cA_{ji}$, preserves weights, and sends a state $a$ to its transpose
$a^\top$.

For every ordered triple $(i,j,k)$, let
$\mathcal T_{ijk}\subseteq\mathcal A_{ij}\times
\mathcal A_{jk}\times\mathcal A_{ki}$ be the family of retained
triads.  A triad $(a,b,c)\in\mathcal T_{ijk}$ creates the cyclic
transitions $a\to b\to c\to a$.

Reversal is an anti-automorphism of the transition digraph.  Thus, if
$C$ is a strongly connected component, then
$C^\top:=\{a^\top:a\in C\}$ is also a strongly connected component.
Throughout Sections~6 and~7, we identify $C$ and $C^\top$, and by a
transition component we mean the reversal class
$C\cup C^\top$, which may consist of one or two strongly connected
components.  Let $\mathcal C$ denote the set of these reversal classes.

For $C\in\mathfrak C$, define
$w_C(ij):=\sum_{a\in\cA_{ij}\cap C}w(a)$.  By construction, $w_C(ij)=w_C(ji)$ for every $C\in\mathcal C$ and $i,j\in I$.
For
$a\in\cA_{ij}$, define its extension weight by
\[
    \ext(a):=
    \sum_{k\in I}\lambda_k
    \sum_{\substack{b\in\cA_{jk},\,c\in\cA_{ki}\\
                    (a,b,c)\in\cT_{ijk}}}
        w(b)w(c).
\]
We call the system \emph{$(p,\alpha,\eta)$-admissible} if
\begin{align}
    \sum_{(a,b,c)\in\cT_{ijk}}w(a)w(b)w(c)
        &\ge p-\eta
        &&\text{for all }i,j,k\in I,
        \label{eq:state-density}\\
    \ext(a)&\ge\alpha-\eta
        &&\text{for every retained state }a.
        \label{eq:state-extension}
\end{align}

Thus \Cref{lem:component-coloured-slice}, applied with target error
$\eta$, produces precisely a $(p,\alpha,\eta)$-admissible state
system.

For $C\in\mathfrak C$, write
\[
    d_C(i):=\sum_{j\in I}\lambda_jw_C(ij),
    \qquad
    S_C:=\{i\in I:d_C(i)>0\}.
\]
We say that $C$ is \emph{spanning} if $S_C=I$, and
\emph{pair-complete} if $w_C(ij)>0$ for every $i,j\in I$.

We shall also use the reduced triad $3$-graph $R_C$ of a transition
component $C$.  Its vertices are the cluster indices, and its edges
are the cluster triples supporting retained triads of any strongly
connected representative of $C$.  This is well defined: distinct
representatives are transposes of one another, and reversal preserves
the underlying cluster triple.  Any two representatives also have the
same period, and we call $C$ aperiodic if one, equivalently every,
strongly connected representative is aperiodic.  Repeated cluster
indices are interpreted through the tagged-copy convention from
\Cref{lem:component-coloured-slice}. For an edge $e$
of $R_C$, let $\mathbf 1_e$ be its incidence vector, with multiplicities,
and define the matching cone
\[
    \cK(R_C):=
    \left\{
        \sum_{e\in E(R_C)}z_e\mathbf 1_e:
        z_e\ge0
    \right\}.
\]
Thus the cluster-weight vector
$\boldsymbol\lambda:=(\lambda_i)_{i\in I}$ lies in $\cK(R_C)$
precisely when $R_C$ has a perfect fractional matching with respect to
the cluster weights.

The state-level conditions immediately imply the following projected
inequalities.

\begin{lemma}[Projection to transition components]
\label{lem:state-projection}
Every $(p,\alpha,\eta)$-admissible state system satisfies
\[
    \sum_{C\in\mathfrak C}
        w_C(ij)w_C(jk)w_C(ki)\ge p-\eta
\]
for all $i,j,k\in I$, and, whenever $w_C(ij)>0$,
\[
    \sum_{k\in I}\lambda_kw_C(ik)w_C(jk)\ge\alpha-\eta.
\]
Moreover, for every $C\in\mathfrak C$, the graph formed by the pairs
$ij$ with $w_C(ij)>0$ is connected on its non-isolated vertices.
\end{lemma}

\begin{proof}
The three states of every retained triad form a directed $3$-cycle
and hence belong to the same transition component.  The total weight
of the retained triads on $(i,j,k)$ is therefore at most
$\sum_Cw_C(ij)w_C(jk)w_C(ki)$, and
\eqref{eq:state-density} gives the first assertion.

Now suppose that $w_C(ij)>0$ and choose
$a\in\mathcal A_{ij}\cap C$.  Every retained extension of $a$ uses
two states in the same strongly connected component as $a$, and hence
in the same reversal class $C$.  Therefore
\[
    \ext(a)\le\sum_k\lambda_k w_C(jk)w_C(ki)
             =\sum_k\lambda_k w_C(jk)w_C(ik),
\]
where the equality follows from $w_C(ki)=w_C(ik)$.  Now (4) gives
the second assertion.

Finally, each strongly connected representative of $C$ has connected
cluster-pair support: a directed transition path projects to a path of
cluster pairs in which consecutive pairs share a cluster.  The two
representatives of a reversal class have the same underlying
undirected cluster-pair support.  Hence the graph formed by the pairs
$ij$ with $w_C(ij)>0$ is connected on its non-isolated vertices.
\end{proof}

We can now state the finite structural theorem that drives the proof
below density $1/3$.  Its proof is postponed to
\Cref{sec:canonical-proof}.

\begin{theorem}[Canonical component theorem]
\label{thm:canonical}
Fix $p\in(0,1/3]$ and $\alpha>\thr(p)$, and fix the cleaning bounds
described above.  Then there exists $\eta_0>0$ such that, for every
$0\le\eta\le\eta_0$, every cleaned
$(p,\alpha,\eta)$-admissible state system has a canonically selected
component $G$ with the following properties.
\begin{enumerate}[label=\textup{(C\arabic*)},leftmargin=2.8em]
    \item
    $G$ is spanning.

    If $p\le p_\star$, then $G$ is the unique spanning component.
    Moreover, either $G$ is the only component with nonempty support,
    or there is a proper set $J\subsetneq I$ such that
    $S_D\subseteq J$ for every component $D\ne G$.

    If $p>p_\star$, then $G$ is the unique pair-complete component.
    Moreover, every spanning component is pair-complete, and at most
    one component $D\ne G$ has nonempty support.

    \item
    The reduced triad graph $R_G$ has a perfect fractional matching.
    Moreover, there exists $\rho>0$, depending only on $p$, $\alpha$ and
    the fixed cleaning bounds, such that every non-negative demand vector
    $\mathbf d=(d_i)_{i\in I}$ satisfying
    $|d_i-\lambda_i|\le\rho\lambda_i$ for every $i\in I$ belongs to
    $\cK(R_G)$.

    \item
    The canonical transition component $G$ is aperiodic.

    \item
    For two one-root deletions constructed from the common joint
    refinement in \Cref{lem:component-coloured-slice}, the two
    canonical components share a positive joint state on the common
    clusters.
\end{enumerate}
\end{theorem}

The proof of \Cref{thm:canonical} is given in
\Cref{sec:canonical-proof}.  For the remainder of this section we
assume the theorem and use the component colours retained by
\Cref{lem:component-coloured-slice} to lift the canonical reduced
component to an actual tight component of the original $3$-graph.

\subsection{Lifting the canonical component}
\label{subsec:canonical-lifting}

Let $G$ be the canonical component supplied by
\Cref{thm:canonical}.  Every retained state carries an actual
tight-component colour.  Moreover, whenever two states are joined by
a transition, the retained triad producing that transition consists
of three states carrying the same actual colour.  Choose a strongly connected representative $G_0$ of $G$.  All states
of $G_0$ carry the same actual component colour.  The states of
$G_0^\top$ carry the same colour, since transposition does not change
the underlying pair.  Hence all states of the reversal class $G$
carry one common actual component colour.  Let $\mathcal C$ be this
actual tight component.

We prove that $C$ satisfies the four conditions required in the
Hamilton framework.

\begin{theorem}[Component-aware lifting]
\label{thm:component-lifting}
Fix $p\in(0,1/3]$ and $\alpha>\thr(p)$.  Let the state systems of a
$3$-graph and of the one-root deletions required in the Hamilton
framework be obtained from
\Cref{lem:component-coloured-slice}.  If their canonical components
satisfy \textup{(C1)}--\textup{(C4)}, then the corresponding actual
tight components can be chosen so that they satisfy
\textup{(F1)}--\textup{(F4)}.  Moreover, the actual component selected
for a fixed $3$-graph is independent of every bounded refinement used
in the regularisation.
\end{theorem}

We first record how an actual component colour is seen by the reduced
system.  For an actual tight component $D$ and a vertex $x$, write $d_{\partial D}(x):=|\{y:xy\in\partial D\}|$. 
For a vertex $v$, put
$X_D(v):=\{x\ne v:vx\in\partial D\}$.  If $x\in X_D(v)$, then every
$y\in N_H(vx)$ satisfies $vy,xy\in\partial D$.  Consequently,
$N_H(vx)\subseteq X_D(v)$, so every nonempty root-colour class has
size at least $\alpha n$, and
$d_{\partial D}(x)\ge d_H(vx)\ge\alpha n$ for every
$x\in X_D(v)$.  The following lemma gives a slightly more general
criterion ensuring that an actual component colour is visible in the
reduced system.

\begin{lemma}[Realisation of actual colours]
\label{lem:actual-colour-realisation}
Let $D$ be an actual tight component and let $U\subseteq V(H)$ satisfy
$d_{\partial D}(x)\ge\xi n$ for every $x\in U$, where $\xi>0$ is
fixed.  If
$|U\cap V_i|\ge\chi|V_i|$ for some fixed $\chi>0$, then $i$ belongs
to the support of a retained transition component carrying the actual
colour $D$.
\end{lemma}

\begin{proof}
Put $U_i:=U\cap V_i$.  Counting ordered incidences from $U_i$ to the
nonexceptional vertices gives
\[
    \sum_{x\in U_i}
       |N_{\partial D}(x)\setminus V_0|
    \ge
    \chi(\xi-\zeta)n|V_i|.
\]
Consequently, for some $j\in I$, the ordered $D$-coloured pair density
from $V_i$ to $V_j$ is bounded below by a positive constant depending
only on $\chi$ and $\xi$.

The component-coloured refinement is performed before the cleaning.
Since the cleaning deletes only a sufficiently small pair mass on each
ordered cluster pair, some retained state over $(i,j)$ still carries
the colour $D$.  Hence $i$ lies in the support of the transition
component containing this state.
\end{proof}

For a fixed vertex $v$, the nonempty root-colour classes $X_D(v)$
partition $V(H)\setminus\{v\}$.  Since every such class has size at
least $\alpha n$, there are at most $1/\alpha$ nonempty root colours.
Consequently, for every cluster $V_i$, some root colour $D$ satisfies
\[
    |X_D(v)\cap V_i|
    \ge
    \frac{\alpha}{2}|V_i|
\]
for all sufficiently large $n$.  Applying
\Cref{lem:actual-colour-realisation} with
$U=X_D(v)$, $\xi=\alpha$ and $\chi=\alpha/2$, we see that $i$ belongs
to the support of a retained transition component carrying colour
$D$.  Hence the supports arising from the root colours together cover
$I$.

We now combine this observation with \textup{(C1)} to lift the
spanning property of the canonical reduced component to the actual
tight component.

\begin{lemma}[Spanning]
\label{lem:lift-F1}
The actual tight component $C$ carrying the canonical reduced
component $G$ is spanning.  In particular, $C$ satisfies
\textup{(F1)}.
\end{lemma}

\begin{proof}
Suppose that $v\notin V(C)$.  Every pair through $v$ then has an
actual component colour different from $C$.  By the root-colour application above, the supports of the
corresponding noncanonical reduced components together cover $I$.

Suppose first that $p\le p_\star$.  By \textup{(C1)}, all
noncanonical components are supported inside one proper set
$J\subsetneq I$.  Their supports therefore cannot cover $I$, a
contradiction.

Now suppose that $p>p_\star$.  By \textup{(C1)}, at most one
noncanonical component has nonempty support.  Since the root-colour
supports cover $I$, this component must be spanning.  Again by
\textup{(C1)}, every spanning component is pair-complete.  We have
therefore obtained a second pair-complete component, contradicting the
uniqueness of $G$.

Hence no such vertex $v$ exists, and $C$ is spanning.
\end{proof}

The preceding lemma establishes \textup{(F1)}.  We next record an
auxiliary consequence which will be needed when the regular partition
is refined in the proof of \textup{(F2)} and when two one-root
deletions are compared in \textup{(F4)}: the actual component selected
above is intrinsic to $H$ and does not depend on the chosen
regularisation.

\begin{lemma}[Invariance of the actual canonical component]
\label{lem:actual-canonical-invariance}
The $3$-graph $H$ has a unique spanning actual tight component,
denoted by $\cC(H)$.  Every component-coloured cleaned regularisation
of $H$, including every prescribed bounded refinement, has its
canonical reduced component inside $\cC(H)$.
\end{lemma}

\begin{proof}
By \Cref{lem:lift-F1}, every component-coloured
regularisation selects an actual spanning component.  Let $C$ be the
component selected by one such regularisation, and suppose that
$D\ne C$ is another actual spanning component.

Since $D$ is spanning, for every vertex $x$ there is a pair
$xy\in\partial D$.  Every edge containing $xy$ belongs to $D$, and
hence
$d_{\partial D}(x)\ge d_H(xy)\ge\alpha n$.
By \Cref{lem:actual-colour-realisation}, the supports of the retained
transition components carrying colour $D$ therefore cover $I$.

If $p\le p_\star$, these pieces are noncanonical and hence, by
\textup{(C1)}, are all supported inside one proper set
$J\subsetneq I$, a contradiction.

If $p>p_\star$, \textup{(C1)} allows at most one noncanonical
component with nonempty support.  It must therefore be spanning, and
hence pair-complete by the same item.  This contradicts the uniqueness
of the canonical pair-complete component.

Thus $C$ is the unique actual spanning component.  The same argument
applies after every bounded prescribed refinement, so all such
regularisations select the same actual component $\cC(H)$.
\end{proof}

We next transfer the robust fractional matching in \textup{(C2)} to
the actual component.  The only additional ingredient is the
vertex-lower-regularity supplied by
\Cref{lem:component-coloured-slice}.

\begin{lemma}[Averaging a cover]
\label{lem:cover-averaging}
For every $\gamma>0$, provided that the regularity error is
sufficiently small, the following holds.  Let a retained canonical
triad have tagged vertex classes $U_1,U_2,U_3$, and let
$c:U_1\cup U_2\cup U_3\to[0,1]$ be a fractional vertex cover of its
actual edges.  Put
$a_j:=|U_j|^{-1}\sum_{v\in U_j}c(v)$ for $j\in[3]$.  Then
$a_1+a_2+a_3\ge1-\gamma$.

The same conclusion holds for repeated cluster indices, with the
corresponding multiplicities.
\end{lemma}

\begin{proof}
Suppose that $a_1+a_2+a_3<1-\gamma$.  Choose independent uniformly
random vertices $X_j\in U_j$ and put
$S:=c(X_1)+c(X_2)+c(X_3)$.  Since $\mathbb E S<1-\gamma$, a positive
proportion, depending only on $\gamma$, of the triples satisfy
$S\le1-\gamma/2$.

Partition $[0,1]$ into intervals of length at most $\gamma/12$.
One product of three level sets then has fixed positive relative
measure and has the property that the sum of the three upper endpoints
is smaller than one.  By the vertex-lower-regularity supplied by
\Cref{lem:component-coloured-slice}, these three level sets contain an
actual canonical edge.  The cover values on this edge sum to less than
one, a contradiction.

For repeated indices, use the coherent tagged copies from
\Cref{lem:component-coloured-slice}; the same argument applies with
multiplicity.
\end{proof}

We can now lift \textup{(C2)}.

\begin{lemma}[Fractional matching]
\label{lem:lift-F2}
The actual canonical component $\cC(H)$ has a perfect fractional
matching.  In particular, it satisfies \textup{(F2)}.
\end{lemma}

\begin{proof}
Put $C:=\cC(H)$ and suppose that $C$ has no perfect fractional
matching.  By linear-programming duality, $C$ has a fractional vertex
cover $c:V(H)\to[0,1]$ with
$\sum_{v\in V(H)}c(v)<n/3$.

We may assume that $\min_v c(v)=0$.  Indeed, if
$m:=\min_v c(v)<1/3$, then
$c'(v):=\min\{(c(v)-m)/(1-3m),1\}$ is again a fractional vertex cover
and still has total weight less than $n/3$.  Replacing $c$ by $c'$,
choose $x$ with $c(x)=0$.

Since $C$ is spanning, choose $u$ with $xu\in\partial C$.  Every edge
of $H$ containing $xu$ belongs to $C$, and hence
$|N_C(xu)|\ge\alpha n$.  For every $z\in N_C(xu)$, the pair
$xz$ also belongs to $\partial C$ and therefore has at least
$\alpha n$ extensions in $C$.  Since every edge of the link $L_C(x)$
is counted at most twice, for all sufficiently large $n$ we have
$e(L_C(x))\ge\alpha^2n^2/3$.

Let $Z:=\{v:c(v)\ge1/2\}$.  Every edge $yz$ of $L_C(x)$ meets $Z$,
because $c(y)+c(z)\ge1$.  Since at most $|Z|n$ unordered pairs meet
$Z$, it follows that $|Z|\ge\tau n$, where
$\tau:=\alpha^2/3$.

Let $\rho>0$ be the uniform demand radius supplied by
\textup{(C2)}, and choose $0<\sigma\le\min\{\rho,1/2\}$.
Choose $\gamma>0$ and the exceptional-set parameter $\zeta>0$
sufficiently small that
\[
    \frac{1}{3(1-\zeta)(1-\gamma)}
    <
    \frac13+\frac{\sigma\tau}{24}.
\]
Choose the parameters in
\Cref{lem:component-coloured-slice} so that
\Cref{lem:cover-averaging} applies with this value of $\gamma$.

Apply \Cref{lem:component-coloured-slice} again with these parameters,
now prescribing the initial partition
$Z\dot\cup(V(H)\setminus Z)$.  By
\Cref{lem:actual-canonical-invariance}, the canonical reduced
component of this refined system is still carried by $C$.  Write
$V_0,V_1,\ldots,V_t$ for the resulting partition, put
$N:=n-|V_0|$ and $\lambda_i:=|V_i|/N$, and let $G$ be its canonical
reduced component.

For each $i\in[t]$, put
$a_i:=|V_i|^{-1}\sum_{v\in V_i}c(v)$.  By
\Cref{lem:cover-averaging},
$\sum_{i\in e}a_i\ge1-\gamma$ for every edge $e$ of $R_G$,
with multiplicity when an index is repeated. Hence
$y_i:=\min\{a_i/(1-\gamma),1\}$ defines a fractional vertex cover of
$R_G$.

Since the partition refines
$Z\dot\cup(V(H)\setminus Z)$ and $|V_0|\le\zeta n$, choosing
$\zeta\ll\tau$ ensures that the clusters contained in $Z$ have total
weight at least $\tau/2$.  For each such cluster, $y_i\ge1/2$, and
therefore
\[
    \sum_i\lambda_i
        \left|y_i-\frac13\right|
    \ge\frac{\tau}{12}.
\]

Put $w_i:=y_i-1/3$ and
$d_i:=\lambda_i(1-\sigma\operatorname{sgn}(w_i))$.
Then $|d_i-\lambda_i|\le\rho\lambda_i$ for every $i$, so
$\mathbf d\in\cK(R_G)$ by \textup{(C2)}.  Thus
$\mathbf d=\sum_{e\in E(R_G)}z_e\mathbf1_e$ for some $z_e\ge0$.
Since $\mathbf y$ is a fractional vertex cover of $R_G$,
\[
\begin{aligned}
    0
    &\le
    \sum_i d_i\left(y_i-\frac13\right)\\
    &=
    \sum_i\lambda_i\left(y_i-\frac13\right)
    -
    \sigma\sum_i\lambda_i
        \left|y_i-\frac13\right|.
\end{aligned}
\]
Consequently,
\[
    \sum_i\lambda_i y_i
    \ge
    \frac13+\frac{\sigma\tau}{12}.
\]

On the other hand, the definition of $y_i$, the bound
$|V_0|\le\zeta n$, and the fact that
$\sum_vc(v)<n/3$ give
\[
\begin{aligned}
    \sum_i\lambda_i y_i
    &\le
    \frac{1}{1-\gamma}\sum_i\lambda_i a_i\\
    &<
    \frac{1}{3(1-|V_0|/n)(1-\gamma)}\\
    &\le
    \frac{1}{3(1-\zeta)(1-\gamma)}\\
    &<
    \frac13+\frac{\sigma\tau}{24},
\end{aligned}
\]
contradicting the preceding lower bound.  Hence $C$ has a perfect
fractional matching, proving \textup{(F2)}.
\end{proof}

We next consider aperiodicity.

\begin{lemma}[Aperiodicity]
\label{lem:lift-F3}
The actual canonical component $\cC(H)$ contains a closed tight walk
whose order is congruent to $1$ modulo $3$.  In particular, it
satisfies \textup{(F3)}.
\end{lemma}

\begin{proof}
By \textup{(C3)}, choose an aperiodic strongly connected
representative $G_0$ of $G$.  Hence it contains a closed directed walk of
some bounded length $L$ with $L\equiv1\pmod3$.  Every transition of
this walk comes from a retained canonical triad, so the walk determines
a fixed finite state pattern.  By
\Cref{lem:finite-state-pattern}, this pattern has a homomorphic
embedding in the actual component carrying the canonical states,
namely $\cC(H)$.  Its image is a closed tight walk of order
$1$ modulo $3$, as required.
\end{proof}

It remains to verify consistency.

\begin{lemma}[Consistency]
\label{lem:lift-F4}
Let $J-x$ and $J-y$ be two members of the local family.  Then
$\cC(J-x)\cup\cC(J-y)$ is tightly connected in $J$.  In particular,
the canonical actual component assignment satisfies \textup{(F4)}.
\end{lemma}

\begin{proof}
Construct the state systems of $J-x$ and $J-y$ using the common joint
refinement from \Cref{lem:component-coloured-slice}.  By
\textup{(C4)}, their canonical reduced components share a positive
joint state.

By \Cref{lem:actual-canonical-invariance}, these two reduced
components are carried by the fixed actual components $\cC(J-x)$ and
$\cC(J-y)$, independently of the chosen joint refinement.  A positive
joint state contains an actual pair $uv$ lying in the shadows of both
components.  Hence there is an edge of $\cC(J-x)$ containing $uv$ and
an edge of $\cC(J-y)$ containing $uv$.  These two edges are tightly
adjacent in $J$, and therefore
$\cC(J-x)\cup\cC(J-y)$ is tightly connected.
\end{proof}

Together,
\Cref{lem:lift-F1,lem:lift-F2,lem:lift-F3,lem:lift-F4}
prove \Cref{thm:component-lifting}.

\subsection{The local Hamilton framework}
\label{subsec:local-hamilton-framework}

We are now ready to prove the local framework theorem in
\Cref{subsec:below-third-reduction}.

\begin{proof}[Proof of \Cref{thm:local-framework}]
Fix $p\in(0,1/3]$ and $\alpha>\thr(p)$.  Put
$\Delta:=\alpha-\thr(p)>0$, and choose $\delta>0$ such that
$4\delta<\Delta$.  Let $\bar\alpha:=\alpha-2\delta$, so that
$\bar\alpha>\thr(p)$.  Choose $\nu>0$ sufficiently small in terms of
$\delta$, and then choose the regularity and cleaning parameters
sufficiently small.

Let $\cP_s$ be the family of $s$-vertex $(p,\nu)$-dense $3$-graphs
with minimum codegree at least $(\alpha-\nu)s$, where $s$ is
sufficiently large.  Fix $F\in\cP_s$. 
Since $\alpha-\nu>\bar\alpha$, choose $\eta>0$ sufficiently small,
with the cleaning bounds fixed as above, so that
\Cref{thm:canonical} applies.  Then
\Cref{lem:component-coloured-slice} gives a cleaned
$(p,\bar\alpha,\eta)$-admissible state system for $F$, and
\Cref{thm:canonical} supplies its canonical component.

By \Cref{thm:component-lifting}, the canonical component determines an
actual tight component $\cC(F)$, independently of the chosen
regularisation, and the assignment $F\mapsto\cC(F)$ satisfies
\textup{(F1)}--\textup{(F4)}.  Hence $\cP_s$ admits a Hamilton
framework, proving \Cref{thm:local-framework}.
\end{proof}

\section{Proof of the canonical component theorem}
\label{sec:canonical-proof}

In this section we prove \Cref{thm:canonical}, the structural result
underlying the lifting argument of the previous section.  Fix
$p\in(0,1/3]$ and $\alpha>\thr(p)$, put $r:=1-p$, and for a transition
component $C$ define
$K_C(i,j):=\sum_k\lambda_k w_C(ik)w_C(jk)$.  By
\Cref{lem:state-projection}, $w_C(ij)>0$ implies
$K_C(i,j)\ge\alpha$.

For \textup{(C1)}, \textup{(C3)}, and \textup{(C4)}, we first work at
zero error, meaning that the admissibility inequalities hold with
$\eta=0$.  We use the following finite compactness observation to pass
to sufficiently small $\eta$.  Here the cleaning bounds fixed in \Cref{thm:canonical} remain fixed
while $\eta\to0$.  Consequently, the cleaned systems have uniformly
bounded complexity, and every retained cluster and state has weight
bounded below by a fixed positive constant.  Hence, from any sequence
$\eta_m\to0$, after relabelling and passing to a subsequence, we may
assume that the cluster set, the retained state sets, the reversal
maps, and the retained triad incidences are fixed, while all cluster
and state weights converge.  For \textup{(C4)}, the same applies
simultaneously to the two systems and their common joint refinement.

The admissibility inequalities pass to the limit and give a zero-error
admissible system.  Since all retained positive weights are bounded
away from zero, the component supports, pair-completeness, transition
components, and their periods are unchanged along the subsequence;
the same is true of the existence or non-existence of a common
positive joint state.  Thus any failure of \textup{(C1)},
\textup{(C3)}, or \textup{(C4)} for arbitrarily small $\eta$ would
persist in a zero-error limit.  We shall use this compactness passage
without further comment.  The robust demand statement in
\textup{(C2)} is quantitative and will instead be proved directly in
\Cref{subsec:canonical-C2}.

\subsection{Spanning and canonical selection: \textup{(C1)}}
\label{subsec:canonical-C1}

We begin with a simple consequence of the projected codegree
condition.  If $i\in S_C$, choose $j$ with $w_C(ij)>0$.  Then
$\alpha\le K_C(i,j)\le d_C(i)\le\lambda(S_C)$.  Since
$\sum_C d_C(i)=1$, every index belongs to at most
$\lfloor1/\alpha\rfloor$ component supports.  Moreover, the projected
density condition shows that every set of at most three indices is
contained in the support of some component.

\begin{lemma}[Spanning component]
\label{lem:spanning-component}
Every zero-error $(p,\alpha)$-admissible state system with
$\alpha>\thr(p)$ has a spanning transition component.
\end{lemma}

\begin{proof}
Suppose not.  Choose an inclusion-minimal set $X\subseteq I$ which is
not contained in a component support, and put $m:=|X|$.  By the
observation above, $m\ge4$.  For every $x\in X$, minimality gives a
component $C_x$ such that
$X\setminus\{x\}\subseteq S_{C_x}$ and $x\notin S_{C_x}$.
The components $C_x$ are distinct, so every vertex of $X$ belongs to
$m-1$ component supports.  Hence
$m-1\le\lfloor1/\alpha\rfloor$.

We first record the two product estimates which will be needed.  If
$m=4$, write $X=\{1,2,3,4\}$ and
$F_i:=X\setminus\{i\}$.  Let $\mathfrak C_i$ be the set of components
whose support contains $F_i$, and for
$e\in\binom{F_i}{2}$ put
$A_i(e):=\sum_{C\in\mathfrak C_i}w_C(e)$.  Density on $F_i$ gives
\[
    p\le\prod_{e\in\binom{F_i}{2}}A_i(e).
\]
If $e=jk$, then $e$ lies in precisely two faces $F_i,F_\ell$, and
$\mathfrak C_i\cap\mathfrak C_\ell=\varnothing$.  Thus
$A_i(e)+A_\ell(e)\le1$, so
$A_i(e)A_\ell(e)\le1/4$.  Multiplying the four face inequalities gives
$p^4\le(1/4)^6$, and hence $p\le1/8$.

Now suppose that $m=5$ and every index belongs to at most four
component supports.  Write $X=\{1,2,3,4,5\}$ and let $C_i$ be the
component missing $i$.  The support multiplicity is saturated on $X$,
so no further component has a positive pair incident with a vertex of
$X$.  For $e=jk$, put $x_i(e):=w_{C_i}(e)$.  Then
$\sum_{i\notin\{j,k\}}x_i(e)=1$.  If
$T\in\binom X3$ and $X\setminus T=\{a,b\}$, only $C_a$ and $C_b$
can contribute to a retained triangle on $T$, and therefore
\[
    p\le
    \prod_{e\in\binom T2}\bigl(x_a(e)+x_b(e)\bigr).
\]
Multiplying these inequalities over the ten triples, each edge
$e=jk$ contributes
$(1-x_a(e))(1-x_b(e))(1-x_c(e))\le(2/3)^3$, where
$a,b,c$ are the remaining three indices.  Hence $p\le8/27$.

We now distinguish the three ranges of the threshold.  If
$p\le p_\star$, then $\alpha>1/3$, so every index lies in at most two
supports, contradicting $m\ge4$.

If $p_\star<p\le p_0$, then $\alpha>\thr_2(p)>1/4$, and hence
$m=4$.  The four-point calculation gives $p\le1/8$, contrary to
$p>p_\star>1/8$.

Finally, suppose that $p_0<p\le1/3$.  Since
$\alpha>\thr_3(p)=r^2/2\ge2/9$, we have $m\in\{4,5\}$.  The two
calculations above give $p\le1/8$ or $p\le8/27$.  Both are impossible:
indeed $p_0>8/27$, since $(1-p)^3-p$ is strictly decreasing, is
positive at $8/27$, and vanishes at $p_0$.
\end{proof}

For the uniqueness part we aggregate one transition component against
its complement.  For a symmetric array $u=(u_{ij})$, put
$d_u(i):=\sum_j\lambda_j u_{ij}$,
$K_u(i,j):=\sum_k\lambda_k u_{ik}u_{jk}$, and
$\Phi(a,b,c):=abc+(1-a)(1-b)(1-c)$.

We shall use the elementary comparisons
$\thr(p)>(1-\sqrt p)^2$, $\thr(p)>(1-2p)^2$, and
$(p^2+r^2)/2<p$ whenever $p_\star<p\le1/3$.

\begin{lemma}[Two-colour profile]
\label{lem:two-colour-profile}
Let $C$ be a transition component and put $u_{ij}:=w_C(ij)$.
Then the following hold.
\begin{enumerate}[label=\textup{(\roman*)},leftmargin=2.6em]
    \item $\Phi(u_{ij},u_{jk},u_{ki})\ge p$ for all $i,j,k$.
    Moreover, $u_{ij}>0$ implies $K_u(i,j)\ge\alpha$, while
    $u_{ij}<1$ implies $K_{1-u}(i,j)\ge\alpha$.

    \item If $p>p_\star$, then
    $u_{ij}\in\{0,1\}\cup[p,r]$ for every $i,j$.

    \item If $p>1/4$ and $u_{ij}=1$ for some pair, then $C$ is
    pair-complete.

    \item If $d_u(i)\ge\alpha$ and $1-d_u(i)\ge\alpha$ for every
    $i$, then $\alpha\le1/3$.
\end{enumerate}
\end{lemma}

\begin{proof}
Every retained triad is either entirely contained in $C$ or entirely
outside $C$, since its three states form a directed $3$-cycle.  This
gives \textup{(i)} directly from the projected density and codegree
conditions.

For \textup{(ii)}, suppose that $0<a:=u_{ij}<p$.  We claim that
$u_{ik}u_{jk}<\alpha$ for every $k$.  Otherwise, with
$b:=u_{ik}$, $c:=u_{jk}$ and $x:=\sqrt{bc}\ge\sqrt\alpha$, AM--GM
gives
\[
    \Phi(a,b,c)
    \le ax^2+(1-a)(1-x)^2.
\]
The right-hand side is convex in $x$.  At $x=1$ it equals $a<p$,
while at $x=\sqrt\alpha$ it is also smaller than $p$, using
$\alpha>(1-\sqrt p)^2$ and
$\alpha>(1-2p)^2$.  This contradicts \textup{(i)}.  Hence
$K_u(i,j)<\alpha$, again a contradiction.  Applying the same argument
to $1-u$ excludes values in $(r,1)$.

For \textup{(iii)}, suppose that $u_{xy}=1$ but $u_{vw}=0$.  If the
two pairs meet, a triple containing both gives
$\Phi(1,0,t)=0$.  If they are disjoint, multiplying the four density
inequalities through the cross-pairs gives
\[
    p^4
    \le
    \prod_{e\in E(K_{\{x,y\},\{v,w\}})}u_e(1-u_e)
    \le\left(\frac14\right)^4,
\]
contrary to $p>1/4$.  The repeated-index cases follow from the same
two-triple calculation.  Thus no zero pair exists.

Finally, for \textup{(iv)}, let
$M:=\sum_{i,j,k}\lambda_i\lambda_j\lambda_k
\Phi(u_{ij},u_{jk},u_{ki})$.  By \textup{(i)},
\[
    M
    =
    \sum_{i,j}\lambda_i\lambda_j
    \bigl(
        u_{ij}K_u(i,j)
        +(1-u_{ij})K_{1-u}(i,j)
    \bigr)
    \ge\alpha.
\]
The weighted Goodman identity gives
$M=1-3\sum_i\lambda_i d_u(i)(1-d_u(i))
\le1-3\alpha(1-\alpha)$.
Thus $(1-\alpha)(1-3\alpha)\ge0$, and hence
$\alpha\le1/3$.
\end{proof}

\begin{lemma}[Canonical selection]
\label{lem:canonical-selection}
If $p\le p_\star$, there is at most one spanning component.  If
$p>p_\star$, every spanning component is pair-complete and there is at
most one pair-complete component.
\end{lemma}

\begin{proof}
Suppose first that $p\le p_\star$ and that $C,D$ are two spanning
components.  Put $u_{ij}:=w_C(ij)$.  Since both components are
spanning, the observation at the beginning of the subsection gives
$d_u(i)\ge\alpha$ and $1-d_u(i)\ge d_D(i)\ge\alpha$ for every $i$.
Part \textup{(iv)} of \Cref{lem:two-colour-profile} gives
$\alpha\le1/3$, a contradiction.

Now suppose that $p>p_\star$ and let $D$ be spanning.  Put
$u_{ij}:=w_D(ij)$ and suppose that some pair has zero $D$-mass.
Choose a zero pair $xy$.  If $x\ne y$, let
$x=v_0,v_1,\ldots,v_\ell=y$ be a shortest path in the positive-pair
graph.  Then $\ell\ge2$ and $u_{v_0v_2}=0$, so density on
$v_0v_1v_2$ gives
$(1-u_{v_0v_1})(1-u_{v_1v_2})\ge p>1/4$.
Hence one of these two positive pair masses is smaller than $1/2$.
If $x=y$, choose $b$ with $u_{xb}>0$; the repeated triple $(x,x,b)$
gives $(1-u_{xb})^2\ge p$, and hence $u_{xb}<1/2$.
Thus in either case there is a positive pair $ab$ with
$q:=u_{ab}<1/2$ adjacent to a zero pair.  By part \textup{(ii)} of
\Cref{lem:two-colour-profile}, $q\in[p,1/2)$.

Fix $z$.  If $u_{az}u_{bz}>0$, part \textup{(iii)} shows that neither
factor can equal one, and hence both lie in $[p,r]$.  Since $q<1/2$, the same corner estimate shows that neither
$u_{az}$ nor $u_{bz}$ can be at least $1/2$.  Thus, with
$t:=\sqrt{u_{az}u_{bz}}$, we have $t\le1/2$, and
\[
    p\le\Phi(q,u_{az},u_{bz})
      \le pt^2+r(1-t)^2.
\]
Hence $t\le1-2p$.  Therefore
$K_D(a,b)\le(1-2p)^2<\alpha$, a contradiction.  Thus every spanning
component is pair-complete.

It remains to prove uniqueness.  Suppose that $C,D$ are pair-complete
and put $u_{ij}:=w_C(ij)$.  Since the complement of $u$ contains $D$,
part \textup{(ii)} gives $p\le u_{ij}\le r$ for every pair.  No
triangle can contain one value at most $1/2$ and another at least
$1/2$, since the maximum of $\Phi$ on
$[p,1/2]\times[1/2,r]\times[p,r]$ is at most
$(p^2+r^2)/2<p$.  Hence all entries lie on the same side of $1/2$.
Replacing $u$ by $1-u$ if necessary, assume
$p\le u_{ij}<1/2$ for every pair.

For fixed $i,j,k$, put $x:=\sqrt{u_{ik}u_{jk}}\le1/2$.  As above,
$p\le px^2+r(1-x)^2$, so $x\le1-2p$.  Consequently
$K_u(i,j)\le(1-2p)^2<\alpha$, contradicting
\Cref{lem:two-colour-profile}\textup{(i)}.
\end{proof}

Fix now a spanning component $G$ and put $g_{ij}:=w_G(ij)$.  The
defect graph $P_G$ has the distinct pair $ij$ as an edge whenever
$g_{ij}<1$; a vertex $i$ is marked when $g_{ii}<1$.  A marked
isolated vertex is regarded as a singleton defect component.  A defect
component $J\ne I$ is called a proper pocket.

\begin{lemma}[Pocket structure]
\label{lem:pocket-structure}
Let $G$ be spanning.
\begin{enumerate}[label=\textup{(\roman*)},leftmargin=2.6em]
    \item Every component $D\ne G$ is supported inside one defect
    component of $P_G$.  If $J$ is a proper pocket, then
    $g_{ij}\ge p$ for $i,j\in J$ and
    $\lambda(J)\ge\alpha/r^2$.

    \item If $J_1,J_2$ are distinct proper pockets, then
    $\alpha\le r^2/2$ and $\alpha\le1/3$.

    \item If $p>p_\star$ and $G$ is the unique pair-complete
    component, then distinct noncanonical components have disjoint
    supports, and at most one noncanonical component has nonempty
    support.
\end{enumerate}
\end{lemma}

\begin{proof}
Every positive pair of $D\ne G$ is a defect pair.  By
shadow-connectivity, all positive pairs of $D$ lie in one defect
component, proving the first assertion of \textup{(i)}.  If $J$ is a
proper pocket, choose $k\notin J$.  The pairs $ik$ and $jk$ are pure
$G$-pairs, so density on $ijk$ gives $g_{ij}\ge p$.  Hence every
noncanonical pair mass in $J$ is at most $r$.  Choosing a positive
pair of a noncanonical component inside $J$ gives
$\alpha\le r^2\lambda(J)$.

For \textup{(ii)}, the pocket-mass estimate gives
$1\ge\lambda(J_1)+\lambda(J_2)\ge2\alpha/r^2$, and hence
$\alpha\le r^2/2$.  Choose $x_t\in J_t$ belonging to a noncanonical
support.  Then
$\sum_{z\in J_t}\lambda_z(1-g_{x_tz})\ge\alpha$.  Since the cross-pair
$x_1x_2$ is pure canonical,
\[
    \alpha
    \le K_G(x_1,x_2)
    \le 1-2\alpha,
\]
and therefore $\alpha\le1/3$.

Finally suppose that $p>p_\star$.  By
\Cref{lem:two-colour-profile}, every positive pair mass of a
noncanonical component is at most $r$: a larger nontrivial mass is excluded by
\Cref{lem:two-colour-profile}\textup{(ii)}, while a pure pair would
make that component pair-complete by
\Cref{lem:two-colour-profile}\textup{(iii)}.  If
$i\in S_D\cap S_E$ for distinct $D,E\ne G$, then
$d_D(i),d_E(i)\ge\alpha/r$, while $d_G(i)\ge p$.  Thus
$1\ge p+2\alpha/r$, which gives $\alpha\le r^2/2$, contrary to
$\alpha>\thr(p)\ge r^2/2$.  Hence distinct noncanonical supports are
disjoint.

By \Cref{lem:canonical-selection}, every noncanonical support is
proper.  Moreover, two distinct such supports must lie in distinct
defect components: along a path of defect edges, two consecutive
edges have a common endpoint and hence, by the disjointness just
proved, must belong to the same noncanonical component.  Thus two
noncanonical components with nonempty support would give two proper
pockets, contradicting \textup{(ii)}.
\end{proof}

We can now finish \textup{(C1)}.  By
\Cref{lem:spanning-component}, a spanning component $G$ exists.

If $p\le p_\star$, then $G$ is unique by
\Cref{lem:canonical-selection}.  Since $\alpha>1/3$, every cluster
belongs to at most two component supports, one of which is $S_G=I$.
Thus two defect edges sharing a vertex belong to the same
noncanonical component.  Consequently each defect component is
contained in one noncanonical support.  No defect component can be all
of $I$, since that would produce a second spanning component.
By \Cref{lem:pocket-structure}\textup{(ii)}, there cannot be two
proper pockets.  Hence either $G$ is the only component with nonempty
support, or there is a proper set $J\subsetneq I$ containing the
support of every component $D\ne G$.

If $p>p_\star$, \Cref{lem:canonical-selection} shows that every
spanning component is pair-complete and that $G$ is the unique
pair-complete component.  Part \textup{(iii)} of
\Cref{lem:pocket-structure} gives that at most one component
$D\ne G$ has nonempty support.

This proves \textup{(C1)} at zero error, and hence for all sufficiently
small $\eta$ by the compactness passage above.

\subsection{Fractional matching and robust demands: \textup{(C2)}}
\label{subsec:canonical-C2}

We now turn to the matching property of the canonical component $G$.
Both parts of \textup{(C2)} are most naturally handled on the dual
side.  A fractional cover below the balanced value $1/3$ can be
normalised to have minimum zero, and then a fractional Berge--Tutte
argument produces a rooted separator in $R_G$.  The structure obtained
in \textup{(C1)} restricts how this separator can meet the canonical
component and leads to the sharp space obstruction.  Once this
obstruction is excluded, the same argument gives a uniform gap for
normalised covers, from which the robust demand cone follows by duality.

\begin{lemma}[Cover separator]
\label{lem:cover-separator}
Let $R$ be a finite weighted $3$-graph of total vertex weight one,
and let $c$ be a fractional vertex cover of $R$ with $\min_i c(i)=0$ and
$\sum_i\lambda_i c(i)\le 1/3+\xi$.
After splitting a zero-cover atom into twins and passing to the
zero-weight-root limit, there are a root $x$ and a partition
$A\dot\cup S\dot\cup T$ of the remaining weighted vertex set such that
$L_R(x)$ has no edge joining $A$ to $A\cup T$ and
$\lambda(A)-\lambda(S)\ge 1/3-2\xi$.
In particular, $\lambda(S)\le 1/3+\xi$.
\end{lemma}

\begin{proof}
Choose an atom $i_0$ with $c(i_0)=0$.  First suppose that all vertex
weights are rational.  Replace each weighted atom by a sufficiently
large proportional blow-up, and split the atom $i_0$ into twins.
Choose one of these twins as the root $x$.  Its relative weight tends
to zero as the blow-up size tends to infinity.

The restriction of $c$ to the link $L_R(x)$ is a fractional vertex
cover of this weighted graph, and hence its fractional matching number
is at most $1/3+\xi$. By the fractional Berge--Tutte formula
\cite[Chapter~2]{ScheinermanUllman1997}, there is a set $S$ such that,
writing $A$ for the isolated vertices of $L_R(x)-S$ and $T$ for the
remaining vertices, 
\[
    \lambda(A)-\lambda(S)
    \ge \frac13-2\xi-o(1).
\]
By definition, no link edge joins $A$ to $A\cup T$.

Letting the blow-up size tend to infinity gives a zero-weight root and
removes the $o(1)$ term.  Since the remaining vertices have total
weight one, $\lambda(A)\le1-\lambda(S)$, and therefore
$\lambda(S)\le1/3+\xi$.

For general real weights, approximate them by rational weights and pass
to a limit.  Repeated indices are first separated using the coherent
tagged-copy convention from
\Cref{lem:component-coloured-slice}, so that $L_R(x)$ is an ordinary
weighted graph.  Pair-state weights and triad incidences are copied
under these refinements, and hence the admissibility conditions are
preserved.
\end{proof}

When applying \Cref{lem:cover-separator} to $R_G$, we regard the
zero-weight root $x$ as a tagged clone of its original atom and copy
all state and component data to it.  Thus the support conclusions in
\textup{(C1)} remain valid for $x$, while all weighted identities are
unchanged.

The next lemma is the only place where the first two branches of the
threshold enter the matching argument.  We use the zero-weight-root
separator from \Cref{lem:cover-separator} with $\xi=0$, and write
$s:=\lambda(S)$ and $Q:=A\cup T$.

\begin{lemma}[Space obstruction]
\label{lem:space-obstruction}
Suppose that the separator in \Cref{lem:cover-separator} exists with
$\xi=0$.  If $p\le p_\star$, then $\alpha\le1/3$.  If
$p>p_\star$, then
$\alpha\le(r+4r^2-2r^3)/6$ and $p\le r^3$.
\end{lemma}

\begin{proof}
Suppose first that $p\le p_\star$.  If $G$ is the only component with
nonempty support, then for $a\in A$ and $z\in Q$ the absence of a
$G$-triad on $xaz$ contradicts the state-level density condition.
Otherwise, by \textup{(C1)} all noncanonical supports lie in one proper
set $J\subsetneq I$.  For every $a\in A$ and $z\in Q$, the separator
excludes a $G$-triad on $xaz$, while density supplies a retained
noncanonical triad.  Hence $x$ and all of $Q$ lie in $J$.  Put
$M:=S\cap J$ and $Y:=I\setminus J$, so $Y\subseteq S$.

Choose $a\in A$ and $y\in Y$.  The pairs $xy$ and $ay$ are pure
canonical, and density on $xay$ implies $w_G(xa)\ge p$.  A canonical
state over $xa$ has no extension through $Q$ by the separator, so its
extension weight is at most $s$.  Therefore
$\alpha\le s\le1/3$.

Now suppose that $p>p_\star$.  By \textup{(C1)}, $G$ is pair-complete.
Part \textup{(ii)} of \Cref{lem:two-colour-profile} gives
$g_{ij}:=w_G(ij)\ge p$ on every pair, and hence every total
noncanonical pair mass is at most $r$.  Fix $a\in A$ and $z_0\in Q$.
Since there is no retained $G$-triad on $xaz_0$, density supplies a
retained state $u$ outside $G$ over $xa$.  Thus
\[
    \alpha\le\ext(u)
    \le
    L:=r^2(1-s)+
       \sum_{z\in S}\lambda_z(1-g_{xz})(1-g_{az}).
\]
There is also a canonical state $u_G$ over $xa$, and the separator
gives
$\alpha\le\ext(u_G)\le
B:=\sum_{z\in S}\lambda_zg_{xz}g_{az}\le s$.
In particular, $\alpha\le1/3$.

For every $z\in Q$, all retained triads on $xaz$ are noncanonical.
The total noncanonical mass on each of their three pairs is at most
$r$, and therefore the density inequality gives $p\le r^3$.

It remains to combine the two extension bounds.  Put
$\theta:=(1+p)/2$.  For $u,v\in[p,1]$, the bilinear expression
$\theta(1-u)(1-v)+(1-\theta)uv$ is at most $r/2$; it is enough to
check the four corners of $[p,1]^2$.  Applying this pointwise on $S$
gives
\[
\begin{aligned}
    \alpha
    &\le \theta L+(1-\theta)B\\
    &\le \theta r^2(1-s)+\frac r2s
     \le \frac23\theta r^2+\frac r6
      =\frac{r+4r^2-2r^3}{6},
\end{aligned}
\]
where the penultimate inequality uses $s\le1/3$ and the coefficient of
$s$ is $rp^2/2>0$.
\end{proof}

We can now obtain both assertions in \textup{(C2)} from the same
separator calculation.  The strict inequality
$\alpha>\thr(p)$ first excludes a fractional cover below $1/3$; after
compactness, it gives a fixed positive gap for every normalised cover.
That gap is exactly what is needed to make the matching cone contain a
relative neighbourhood of the cluster-weight vector.

\begin{lemma}[Robust fractional matching]
\label{lem:robust-fractional-matching}
For all sufficiently small $\eta>0$, the reduced triad graph $R_G$ has
a perfect fractional matching.  Moreover, there is $\rho>0$ such that
every non-negative demand vector $\mathbf d=(d_i)_{i\in I}$ satisfying
$|d_i-\lambda_i|\le\rho\lambda_i$ for every $i$ belongs to
$\cK(R_G)$.
\end{lemma}

\begin{proof}
We first work at zero error.  Suppose that $R_G$ has no perfect
fractional matching.  By linear-programming duality it has a fractional
vertex cover $c$ of weighted size smaller than $1/3$.  If
$m:=\min_i c(i)$, then $m<1/3$ and
$c_0(i):=\min\{(c(i)-m)/(1-3m),1\}$ is again a fractional vertex
cover, has minimum zero, and still has weighted size smaller than
$1/3$.  Apply \Cref{lem:cover-separator} to $c_0$ with $\xi=0$.

If $p\le p_\star$, \Cref{lem:space-obstruction} gives
$\alpha\le1/3=\thr_1(p)$, a contradiction.  If
$p_\star<p\le p_0$, it gives
$\alpha\le(r+4r^2-2r^3)/6=\thr_2(p)$, again a contradiction.  If
$p>p_0$, it gives $p\le r^3$, contrary to the definition of $p_0$.
Thus $R_G$ has a perfect fractional matching in every zero-error
system.

We next claim that there are constants $\xi_0>0$ and $\eta_0>0$ such
that, whenever $0<\eta<\eta_0$, every fractional vertex cover $c$ of
the canonical $R_G$ with $\min_i c(i)=0$ satisfies
$\sum_i\lambda_i c(i)\ge1/3+\xi_0$.  Otherwise, there is a sequence of
cleaned admissible systems with $\eta\to0$ and normalised covers whose
weighted sizes tend to at most $1/3$.  By bounded complexity and the
positive cleaning thresholds, pass to a subsequence in which the
cluster set and all retained incidences are fixed, while the weights and
cover values converge.  The limit is a zero-error admissible system
with a normalised fractional cover of weighted size at most $1/3$.
The first part of the proof forces its weight to be exactly $1/3$;
applying \Cref{lem:cover-separator} with $\xi=0$ and then
\Cref{lem:space-obstruction} gives the same contradiction as above.
This proves the uniform gap.

In particular, for sufficiently small $\eta$, $R_G$ has a perfect
fractional matching: otherwise a cover of weight smaller than $1/3$
could be normalised as above, contradicting the gap.  Let now $c$ be
any fractional vertex cover and write
$\sum_i\lambda_i c(i)=1/3+\varepsilon$ and
$D(c):=\sum_i\lambda_i|c(i)-1/3|$.  Then $\varepsilon\ge0$.  We claim
that $\varepsilon\ge\rho_0D(c)$ for some fixed $\rho_0>0$.

Let $m:=\min_i c(i)$.  If $m\ge1/3$, then $D(c)=\varepsilon$.  If
$m<1/3$, put $a:=1/3-m$ and define
$c'(i):=\min\{(c(i)-m)/(3a),1\}$.  The uniform gap gives
\[
    \frac13+\xi_0
    \le\sum_i\lambda_i c'(i)
    \le\frac13+\frac{\varepsilon}{3a},
\]
so $\varepsilon\ge3\xi_0a$.  If
$N:=\sum_i\lambda_i(1/3-c(i))_+$, then $N\le a$ and
$D(c)=\varepsilon+2N\le\varepsilon+2a$.  Hence
$\varepsilon\ge\rho_0D(c)$, for example with
$\rho_0:=\min\{1/3,\xi_0\}$.

Choose $0<\rho<\rho_0$ and suppose that
$|d_i-\lambda_i|\le\rho\lambda_i$ for every $i$.  For every fractional
cover $c$, writing $w_i:=c(i)-1/3$, we obtain
\[
    \sum_i d_iw_i
    \ge
    \sum_i\lambda_iw_i-
       \rho\sum_i\lambda_i|w_i|
    \ge(\rho_0-\rho)D(c)\ge0.
\]
We claim that these inequalities imply $\mathbf d\in\cK(R_G)$.
Otherwise Farkas' lemma gives a vector $y$ with
$\langle y,\mathbf1_e\rangle\ge0$ for every $e\in E(R_G)$ but
$\langle y,\mathbf d\rangle<0$.  If $\min_i y_i\ge0$, this is
impossible since $\mathbf d\ge0$.  Otherwise let $m:=\min_i y_i<0$.
Then $c_i:=(y_i-m)/(-3m)$ is a fractional vertex cover; truncating it
at one preserves the cover property and only decreases its pairing with
$\mathbf d$.  Applying the preceding inequality to the truncated cover
and then undoing the normalisation gives
$\langle y,\mathbf d\rangle\ge0$, a contradiction.  Therefore
$\mathbf d\in\cK(R_G)$.

The constants $\xi_0$, $\rho_0$ and $\rho$ depend only on
$p$, $\alpha$ and the fixed cleaning bounds, and hence the same
positive value of $\rho$ works for every admissible system under
consideration.
This proves \textup{(C2)}.

\end{proof}

\subsection{Aperiodicity: \textup{(C3)}}
\label{subsec:canonical-C3}

We next turn from the support geometry of $G$ to the internal dynamics
of its transition digraph.  Every retained triad gives a directed
$3$-cycle, so the period of $G$ divides three; the only possible
obstruction is therefore a genuine period-three decomposition.  The two
canonical regimes exclude this obstruction in different ways.  Below
$p_\star$ the support statement in \textup{(C1)} leaves a pure
canonical diagonal outside the unique pocket, whereas above $p_\star$
pair-completeness and reversal symmetry force the sharp bound
$\alpha\le r^2/2$.

\begin{lemma}[Aperiodicity]
\label{lem:canonical-aperiodic}
The canonical transition component $G$ is aperiodic.
\end{lemma}

\begin{proof}
Suppose that $G$ has period three.  Choose a strongly connected
representative $C$ of $G$ with cyclic classes $C_0,C_1,C_2$.
If $C^\top\ne C$, give $C^\top$ the corresponding cyclic classes so
that transposition sends phase $t$ to phase $-t$ modulo three.  If
$C^\top=C$, relabel $C_0,C_1,C_2$ so that the same property holds;
this is possible because transposition reverses every transition.
Let $G_t$ denote the total phase-$t$ part of the reversal class $G$.

Assume first that $p\le p_\star$.  By \textup{(C1)}, either $G$ is the
only component or all noncanonical supports lie in one proper set
$J$.  Hence there is an index $i$ such that the diagonal pair $ii$ is
pure canonical.  Let $x_t$ be the total mass of phase $t$ over $ii$.
If $p>1/9$, then the retained density on $(i,i,i)$ is at most
$3x_0x_1x_2\le1/9$, a contradiction.

Suppose that $p\le1/9$.  Since the retained density on $(i,i,i)$ is
positive, each phase occurs on $ii$.  Choose a state $a_t$ of phase
$t$ over $ii$.  If $x^t_{ik}$ denotes the phase-$t$ mass over the
ordered pair $(i,k)$, then
\[
    \alpha\le \ext(a_t)
    \le\sum_k\lambda_k x^{t+1}_{ik}x^{t+2}_{ki},
\]
with indices modulo three.  Summing over $t$ and using
$\sum_t x^t_{ik}\le1$ and $\sum_t x^t_{ki}\le1$ gives
$3\alpha\le1$, contrary to $\alpha>1/3$.  Thus the first regime is
aperiodic.

Now suppose that $p>p_\star$.  For each ordered pair put
\[
    x^t_{ij}:=\sum_{a\in\mathcal A_{ij}\cap G_t}w(a),
    \qquad
    y_{ij}:=1-x^0_{ij}-x^1_{ij}-x^2_{ij}.
\]
By the phase convention,
$x^t_{ji}=x^{-t}_{ij}$ and $y_{ji}=y_{ij}$.

We first claim that $y_{ii}>0$ for every $i$.  Otherwise every state
over $ii$ belongs to $G$, and the retained density on $(i,i,i)$ is at
most $3x^0_{ii}x^1_{ii}x^2_{ii}\le1/9<p$, a contradiction.  Next,
$x^1_{ii}=x^2_{ii}=0$.  Indeed, if $x^1_{ii}>0$, then reversal gives
$x^2_{ii}>0$.  A phase-$1$ state over $ii$ has extension at most
$\sum_k\lambda_kx^2_{ik}x^0_{ik}$, while an outside state over $ii$
has extension at most $\sum_k\lambda_ky_{ik}^2$.  Since pointwise
$4x^0_{ik}x^2_{ik}+y_{ik}^2\le1$, we obtain $5\alpha\le1$, contrary
to $\alpha>\thr(p)\ge2/9$.  Pair-completeness now gives
$x^0_{ii}>0$, and \Cref{lem:two-colour-profile}\textup{(ii)} gives
$x^0_{ii}\le r$.

No pair is pure canonical.  Otherwise, for some distinct $i,j$ we
have $y_{ij}=0$.  The repeated triples $(i,i,j)$ and $(j,j,i)$ give
\[
    p\le x^0_{ii}(x^1_{ij})^2,
    \qquad
    p\le x^0_{jj}(x^2_{ij})^2.
\]
Since the two diagonal masses are at most $r$, we get
$x^1_{ij},x^2_{ij}\ge\sqrt{p/r}$, and hence
$1\ge2\sqrt{p/r}$.  This implies $p\le1/5$, contrary to
$p>p_\star$.  Therefore every canonical pair mass
$s_{ij}:=x^0_{ij}+x^1_{ij}+x^2_{ij}$ lies in $(0,1)$, and another
application of \Cref{lem:two-colour-profile}\textup{(ii)} gives
$s_{ij}\le r$.

Put $d_1(i):=\sum_j\lambda_jx^1_{ij}$.  Reversal gives
\[
    2\sum_i\lambda_i d_1(i)
    =
    \sum_{i,j}\lambda_i\lambda_j
    \bigl(x^1_{ij}+x^2_{ij}\bigr)
    \le r.
\]
Choose $i$ with $d_1(i)\le r/2$.  Since $x^1_{ij}\le r$,
\[
    \sum_j\lambda_j(x^1_{ij})^2\le\frac{r^2}{2}.
\]
A phase-$0$ diagonal state over $ii$ extends only through a phase-$1$
state over $(i,j)$ and its reversed phase-$2$ state over $(j,i)$, so
\[
    \alpha
    \le\sum_j\lambda_jx^1_{ij}x^2_{ji}
    =\sum_j\lambda_j(x^1_{ij})^2
    \le\frac{r^2}{2}.
\]
This contradicts $\alpha>\thr(p)$, since
$\thr_2(p)>r^2/2$ on $(p_\star,p_0]$ and
$\thr_3(p)=r^2/2$ on $(p_0,1/3]$.

This proves \textup{(C3)} at zero error, and hence for all sufficiently
small $\eta$ by the compactness passage above.
\end{proof}

\subsection{Two-root consistency: \textup{(C4)}}
\label{subsec:canonical-C4}

It remains to compare canonical components belonging to two different
one-root deletions.  The common joint refinement places the two state
systems on a common finite refinement, so a failure to share a positive
joint state would separate the states of the first canonical component
from all states of the second.  Aggregating the first component against
its complement then produces exactly the closed two-colour profile
studied in \Cref{subsec:canonical-C1}.  Thus the uniqueness mechanisms
behind \textup{(C1)} also force the required consistency.

\begin{lemma}[Common joint state]
\label{lem:canonical-common-state}
The two canonical components share a positive joint state on the common
clusters.
\end{lemma}

\begin{proof}
Let $G_x$ and $G_y$ be the two canonical components and suppose that
they have no common positive joint state.  On every common ordered
cluster pair, let $U$ be the set of joint states whose first projection
belongs to $G_x$, and let $u_{ij}$ be its total mass.  Every retained triad on the common
clusters is either entirely contained in $U$ or disjoint from $U$,
while the complement contains all states of $G_y$.  Hence the joint
profile satisfies the conclusions of
\Cref{lem:two-colour-profile}, with the complementary cherry bounded
below by the $G_y$-cherry.

If $p\le p_\star$, both canonical components are spanning.  Thus
$d_u(i)\ge\alpha$ and $1-d_u(i)\ge\alpha$ for every common cluster,
and \Cref{lem:two-colour-profile}\textup{(iv)} gives
$\alpha\le1/3$, a contradiction.

Suppose now that $p>p_\star$.  Both canonical components are
pair-complete, so $0<u_{ij}<1$ on every common pair.  Part
\textup{(ii)} of \Cref{lem:two-colour-profile} gives
$p\le u_{ij}\le r$.  A triple cannot contain one pair mass at most
$1/2$ and another at least $1/2$, since the same corner estimate used in
\Cref{lem:canonical-selection} gives
$\Phi(a,b,c)\le(p^2+r^2)/2<p$.  Hence all pair masses lie on the same
side of $1/2$.  Replacing $u$ by $1-u$ if necessary, assume
$p\le u_{ij}<1/2$ for every pair.

Fix $i,j,k$ and put $t:=\sqrt{u_{ik}u_{jk}}\le1/2$.  By AM--GM and
the closed density inequality,
\[
    p
    \le\Phi(u_{ij},u_{ik},u_{jk})
    \le p t^2+r(1-t)^2.
\]
Hence $t\le1-2p$.  Averaging over $k$ gives
$K_u(i,j)\le(1-2p)^2$, while $u_{ij}>0$ gives
$K_u(i,j)\ge\alpha$.  Therefore
$\alpha\le(1-2p)^2<\thr(p)$, again a contradiction.
Thus the two canonical components share a positive joint state.
\end{proof}

This proves \textup{(C4)} at zero error, and hence for all sufficiently
small $\eta$ by the compactness passage above. Together with
\textup{(C1)}--\textup{(C3)}, it completes the proof of
\Cref{thm:canonical}.

\section{Concluding remarks}
\label{sec:concluding}

In this paper, we have determined the minimum-codegree threshold for
tight Hamilton cycles in linearly quasirandom $3$-graphs throughout
the full density range $p\in(0,1)$. Our proof uses a common ordered pair-state regularity
language, but the relevant reduced structure changes at the density
barrier $1/3$.  Above $1/3$, the whole state digraph is primitive and
arbitrary endpoint pairs can be connected.  At and below $1/3$, the
whole digraph may be disconnected, and the correct object is instead a canonical aperiodic transition
component whose actual tight colour is stable under bounded
refinements.

A notable feature of the result is that the threshold function is
discontinuous, with two genuine jumps.  This differs from many familiar
extremal threshold problems in which the governing obstruction varies
continuously with the density parameter.  In the present setting, the
discontinuities correspond to abrupt changes in the extremal mechanism
governing Hamiltonicity, despite the ambient assumption of linear
quasirandomness.  More generally, it would be interesting to investigate
when such discontinuous transitions can occur in quasirandom hypergraph
problems, which structural obstructions are responsible for them, and
whether several distinct transitions can arise within a single threshold
problem.

In the present problem, the two jumps have different origins.  The
point $p_0$, defined by $p_0=(1-p_0)^3$, marks the disappearance of the
fractional-matching, or space, obstruction.  Consequently, the jump at
$p_0$ is caused by the loss of a feasible extremal mechanism rather
than by two threshold curves crossing.  The second jump occurs at
$p=1/3$.  Below this density, Hamiltonicity can be governed by a
suitable canonical tight component even though arbitrary pairs need not
be connectable, whereas above $1/3$ the relevant reduced structure is
the whole pair-state digraph and global pair-connectability becomes
available.  The discontinuity at $1/3$ therefore reflects a genuine
change in the connectivity mechanism underlying the Hamilton cycle.

The component-coloured regular slice and the reduced-to-actual demand
transfer developed here may be useful in other quasirandom spanning
problems where connectivity, fractional matching and periodicity must
be controlled simultaneously.  It would also be interesting to develop
analogous profile theorems for higher uniformity, other minimum
$\ell$-degree conditions, and different notions of hypergraph
quasirandomness, particularly in settings where further phase
transitions may appear.

\section*{Acknowledgements}
The author is grateful to Richard Lang for suggesting that the
Hamiltonicity framework developed in his joint work with
N.~Sanhueza-Matamala \cite{LangSanhuezaMatamala2026} could be useful
for the remaining cases.  This suggestion was made after the first
version of the paper appeared and proved useful in the subsequent
development of the paper.
The author is also grateful to Daniel Kr\'al' for
reading an earlier version of the paper and for several helpful comments.

The author also acknowledges the use of ChatGPT by OpenAI in the early
stages of this project and during the preparation of the manuscript.  It
was used to improve the language, organisation, and presentation of the
manuscript, to assist in revising preliminary drafts, and as an
interactive tool for discussing possible proof strategies and extremal
constructions.  In particular, ideas underlying most of the lower-bound constructions
in the paper arose during interactions with ChatGPT.  The author assumes
full responsibility for all statements, proofs, constructions, citations,
and conclusions presented in this paper.

\bibliographystyle{abbrv}
\bibliography{reference_new}

\end{document}